\documentclass[a4paper,12pt]{article}

\usepackage[english]{babel}
\usepackage[utf8]{inputenc}
\usepackage{etex}
\usepackage{pdfpages}
\usepackage{amsfonts}
\usepackage{amsmath}
\usepackage{amssymb}
\usepackage{amsthm}
\usepackage{appendix}
\usepackage{hyperref}
\hypersetup{colorlinks,linkcolor={blue},citecolor={blue},urlcolor={red}} 
\usepackage{cases}
\usepackage{enumerate}
\usepackage{caption}
\usepackage{setspace}
\usepackage{ulem}
\usepackage{enumitem}
\usepackage{scalerel}

\usepackage{upgreek}
\usepackage{fancyhdr}

\usepackage{manfnt} 
\usepackage{mathrsfs}
\usepackage{minitoc}
\usepackage{natbib}
\usepackage{accents}

\newcommand{\q}[1]{``#1''}
\usepackage{amsmath,amsfonts,bbm,amsthm}
\usepackage{cleveref}       
\usepackage{amssymb}
\usepackage{float}
\usepackage{latexsym}
\usepackage{eucal}
\usepackage{mathtools}
\usepackage{tikz}
\usetikzlibrary{decorations.markings}
\usepackage{pgfplots}
\usepackage{lipsum}
\usepackage{color}
\usepackage{tabularx}
\usepackage{array}
\usepackage[linesnumbered,ruled,vlined]{algorithm2e}
\SetKwInput{KwInput}{Input}
\SetKwInput{KwOutput}{Output}
\SetKwInput{KwEnsure}{Ensure}
\usepackage{graphicx,caption,subcaption}
\usepackage{moreverb}
\usepackage{multirow} 
\usepackage{url} 
\usepackage[all]{xy} 
\usepackage{shorttoc} 
\usepackage{textcomp} 
\usepackage{hyperref} 
\usepackage{url}
\usepackage{enumerate}
\usepackage{enumitem}
\usepackage[right]{eurosym}
\def \R{\mathbb{R}}
\DeclareMathOperator*{\esssup}{ess~sup}
\DeclareMathOperator*{\essinf}{ess~inf}
\def \N{\mathbb{N}}

\usepackage{graphics}
\usepackage{graphicx}
\usepackage{hyperref}

\usepackage{lastpage}
\usepackage{latexsym}

\usepackage{color}
\usepackage{comment}

\usepackage{enumerate}

\usepackage{fancyhdr}
\usepackage[top=3cm, bottom=3cm, left=1.5cm , right=2cm]{geometry}
\newtheorem{Theorem}{Theorem}[section]
\newtheorem{Corollary}[Theorem]{Corollary}
\newtheorem{Lemma}[Theorem]{Lemma}
\newtheorem{Definition}[Theorem]{Definition}
\newtheorem{Proposition}[Theorem]{Proposition}
\newtheorem{Remark}[Theorem]{Remark}

\numberwithin{equation}{section}

\usepackage[framemethod=tikz]{mdframed}
\usetikzlibrary{decorations.pathmorphing,calc}
\newcommand\wavydecor{%
    \draw[decoration={coil,aspect=0.1,segment length=5pt,amplitude=1.0pt},decorate,line width=1.5pt,black]
      (O|-P) -- (O);
}
\newmdenv[
hidealllines=true,
innerleftmargin=10pt,
innerrightmargin=0pt,
innertopmargin=0pt,
innerbottommargin=0pt,
leftmargin=-10pt,
skipabove=.5\baselineskip,
skipbelow=.5\baselineskip,
singleextra={\wavydecor},
firstextra={\wavydecor},
secondextra={\wavydecor},
middleextra={\wavydecor}
]{done}

\usepackage[linecolor=black,textwidth=25mm,textsize= footnotesize]{todonotes} 

{ \begin{list}%
	{$\bullet$}%
	{\setlength{\labelwidth}{10pt}%
	 \setlength{\leftmargin}{35pt}%
	 \setlength{\itemsep}{\parsep}{1pt}}}%
{ \end{list} }

\title{ Utility Maximization Problem with Uncertainty and a Jump Setting\thanks{Acknowledgements:  The authors research is part of the ANR project DREAMeS (ANR-21-CE46-0002) and benefited from the support of the "Chair Risques Emergents en Assurance" under the aegis of Fondation du Risque, a joint initiative by Le Mans University and Cov\'ea.}}  

\author{Sarah Kaaka\"i~ \thanks
{\small Laboratoire Manceau de Math\'ematiques \& FR CNRS N\textsuperscript{o} 2962, Institut du Risque et de l'Assurance, Le Mans University.}
\and Anis Matoussi
\thanks{ \small  Laboratoire Manceau de Math\'ematiques \& FR CNRS N\textsuperscript{o} 2962, Institut du Risque et de l'Assurance, Le Mans University and Ecole Polytechnique.
}\\
 \and Achraf Tamtalini~ \thanks
{\small  Laboratoire Manceau de Math\'ematiques \& FR CNRS N\textsuperscript{o} 2962, Institut du Risque et de l'Assurance, Le Mans University.}
}
\date{}

\begin{document}
 \maketitle

  \abstract{   We study a robust utility maximization problem in the unbounded case with a general penalty term and information including jumps. We focus on time consistent penalties and we prove that there exists an optimal probability measure solution of the robust problem. Then, we characterize the dynamic value process of our stochastic control problem as the unique solution of a Quadratic-Exponential BSDE. 
  }

\bigskip


\section*{Introduction}
One of the major problems in asset pricing is the valuation in incomplete markets. In such markets, the decision maker/agent could use the well known utility maximization approach and the literature is particularly rich on the subject (see for example \citet{rouge2000pricing}, \citet{hu2005utility}, \citet{morlais2009quadratic} and \citet{carmona2008indifference} among many others). However, in many cases, the decision maker does not know the probability distribution (also called prior or model) governing the stochastic nature of the problem she/he is facing. Thus, before solving the utility maximization problem, the decision maker is faced with an intermediate problem of choosing an \q{optimal} probability. This type of problems are called robust utility maximization problems or utility maximization problems under model uncertainty. In the mathematical finance literature, we can find two types of approaches to solve robust utility maximization problems. The first one relies on convex duality methods which are presented in \citet{quenez2004optimal}, \citet{gundel2005robust}, \citet{schied2007optimal} and \citet{schied2005duality}. The second one, which we will follow in this article, is based on a stochastic control approach and the powerful tool of BSDEs.\\
In this article, this uncertainty is captured by considering a set of plausible probability measures that will be penalized through a penalty functional. This penalty functional will measure the distance between any plausible probability $Q$ and the reference/historical one denoted $P$. In \citet{anderson2003quartet} and \citet{hansen2006robust} for example,  a hedging problem was addressed by using the classical entropic penalty under a Markovian setting and hence Hamilton-Jacobi-Bellman (HJB) equations were derived in order to characterize the optimal strategies. The authors in \citet{skiadas2003robust} have followed the same point of view and obtained a BSDE that coincides with the one describing a stochastic differential utility (see also \citet{duffie1992stochastic} and \citet{duffie1994continuous} for more about stochastic differential utilities).\\
More recently, \citet{bordigoni2007stochastic} addressed a robust problem in a more general setting which is non Markovian by using stochastic control techniques. More precisely, they studied the following robust maximization problem:
\begin{equation}\label{bordigoni}
\underset{\pi}{\sup}~\underset{Q \in \mathcal{Q}}{\inf}\mathbb{U}(\pi, Q)
\end{equation}
where $\pi$ runs through a set of strategies and $Q \in \mathcal{Q}$ through a set of models. The simplest case corresponds to the case where the set $\mathcal{Q}$ is the singleton $\{P\}$ and $\mathbb{U}(\pi, P)$ is simply the $P$-expected utility from a non bounded terminal wealth and consumption/investment portfolio. The term $\mathbb{U}(\pi, Q)$ is the sum of $Q$-expected utility and an entropic penalization term. The set $\mathcal{Q}$ is assumed to have certain properties and usually does not need to be specified in any detail. Their work is cast in the case of a continuous filtration and the first minimization problem was solved by proving the existence of a unique optimal probability $Q^*$. They also characterized the value process of the stochastic control problem as the unique solution of a Quadratic BSDE (QBSDE). In the same spirit, \citet{faidi2013robust} studied the same problem using two type of penalties: the first one is the $f$-divergence penalties in the general framework of a continuous filtration and the second one is the time-consistent penalties studied in the context of a Brownian filtration. For the latter, they also characterized the value process as the unique solution of a QBSDE.\\
In this paper, we study the first minimization problem in \eqref{bordigoni} in the case of discontinuous filtration (where the information includes jumps) using time consistent penalties. Note that in our framework, the relative entropic penalty, as we will see further, is a special case of the class of time-consistent penalties. We first start by showing that the minimization problem in \eqref{bordigoni} is well posed and we prove the existence of an optimal probability $Q^*$ using a Koml\'os-type argument. Second, we prove that the value process of the minimization problem is described by a class of Quadratic-Exponential BSDE with jumps (QEBSDEJ) with unbounded terminal condition. We stress that for a given unbounded terminal condition, the study of Quadratic BSDEs is a difficult problem, see for instance \citet{briand2006bsde}, \citet{briand2008quadratic} and \citet{barrieu2013monotone} in the continuous framework and we emphasize that adding jumps to our optimization problem involves significant difficulties in solving the related BSDEs. \citet{karoui2016quadratic} have obtained existence result for this new class of BSDEs with jumps with unbounded terminal condition. However, they have showed uniqueness only in the bounded case. In this paper, we use the convexity of the generator of our BSDE to show the uniqueness of solution of the BSDE by extending the work of \citet{briand2008quadratic} in the Brownian setting.\\
The paper is structured as follows. Section \ref{framework} establishes the general framework, in which we assume the existence of a stochastic basis carrying a Brownian motion and a compensated integer-valued random measure that possesses a weak predictable representation property. In section \ref{estimates}, we give a number of estimates for subsequent use. We then prove with the help of Koml\'os theorem that there exists an optimal probability $Q^*$. Finally, in section \ref{BSDE}, we treat our optimization problem from a stochastic control point of view, and show, thanks to Bellman Optimality Principle, that the corresponding value process is the unique solution of a QEBSDEJ.

\section{Framework of the optimization problem}
\label{framework}
\subsection{Setting and notations}
This section sets out the notation and the assumptions that will be assumed to hold in the sequel. Let $(\Omega, \mathcal{F}, \mathbb{F}, P)$ be a filtered probability space with a finite time horizon $T < \infty$ and a filtration $\mathbb{F} = (\mathcal{F}_t)_{t\in[0, T]}$ satisfying the usual conditions of right continuity and completeness, in which all semimartingales are taken to have right continuous paths with left limits. We assume that that $\mathcal{F}_0$ is trivial and $\mathcal{F} = \mathcal{F}_T$. On this stochastic basis, let $W$ be a $d$-dimensional standard Brownian motion and let $\mu(dt, dx) = (\mu(w, dt, dx)|w \in \Omega)$ denote an integer-random valued measure on $([0, T] \times E, \mathcal{B}([0,T]) \otimes \mathcal{E})$ with compensator $\nu \vcentcolon = \nu^P(w, dt, dx)$ under $P$, where $E \vcentcolon=\R^d \backslash \{0\}$ is equipped with its Borel $\sigma$-field $\mathcal{E} \vcentcolon= \mathcal{B}(E)$.\\
On $(\widetilde{\Omega}, \widetilde{\mathcal{F}}) \vcentcolon= (\Omega \times [0,T] \times E, \mathcal{F} \otimes \mathcal{B}[0, T]\otimes\mathcal{E})$, we define the measure $P\times \nu$ by
\begin{equation}
P \times \nu(\widetilde{B}) = E\left[\int_{[0,T]\times E}\mathbbm{1}_{\widetilde{B}}(w, t, x)\nu(w, dt, dx)\right],~\widetilde{B} \in \widetilde{\mathcal{F}}.
\end{equation}
We denote by $\widetilde{\mathcal{P}}\vcentcolon = \mathcal{P}\otimes \mathcal{E}$ where $\mathcal{P}$ is the predictable $\sigma$-field on $\Omega \times [0,T]$. We say that a function on $\widetilde{\Omega}$ is \textit{predictable} if it is $\widetilde{\mathcal{P}}$-measurable.\\
We will assume that the compensator $\nu$ is absolutely continuous with respect to the $\lambda \otimes dt$ with a density $\xi$:
\begin{equation}
\label{nu}
\nu(w, dt, dx) = \xi_t(w, x) \lambda(dx)dt,
\end{equation}
where $\lambda$ is a $\sigma$-finite  measure on $(E, \mathcal{E})$, that satisfies the following condition: $\int_E 1 \wedge |x|^2 \lambda(dx) < \infty$ and the density $\xi$ is $\widetilde{\mathcal{P}}$-measurable, positive and bounded:
\begin{equation}
\label{zeta}
0 < \xi_t(w, x) \le C_\nu < \infty,~P\times\lambda(dx) \times dt-a.e.~~~\text{for some constant}~C_\nu.
\end{equation}
Note that, thanks to \eqref{nu}, we have that $\nu(\{t\}\times E)=0$ for all $t$, and $\nu([0, T]\times E)\le C_\nu T \lambda(E)$.\\
For $\psi$ a predictable function on $\widetilde{\Omega}$, we define its integral process with respect to $\mu$  as:
\begin{equation}
\label{integral}
(\psi\star	\mu)_t \vcentcolon=
\left\{
\begin{aligned}
 &\int_{[0,t]\times E} \psi_t(x) \mu(w, ds, dx)~~~\text{if} \int_{[0,t]\times E} |\psi_t(w, x)| \mu(w, ds, dx) < \infty,\\
 &+ \infty ~~~~~~~~~~~~~~~~~~~~~~~~~~~~~~~~~~\text{otherwise}.
\end{aligned}\right.
\end{equation}
In the same way, we define the integral process with respect to $\nu$.\\
Let $\widetilde{\mu}$ be the compensated measure of $\mu$
\begin{equation}
\widetilde{\mu}(w, dt, dx) = \mu(w, dt, dx) - \nu(w, dt, dx).
\end{equation}
To alleviate the notations, we will omit the dependence on $w$ in the different stochastic quantities. In the following, we recall some properties that can be found in \citet{becherer2006bounded} or \citet{jacod2013limit}. First, for any predictable function $\psi$, the process $\psi\star \nu$ is a predictable process whereas $\psi \star \mu$ is an optional process. We recall that $E[|\psi|\star \mu_T]=E[|\psi|\star \nu_T]$. If $(|\psi|^2\star \mu)^{1/2}$ is locally integrable, then $\psi$ is integrable with respect to $\widetilde{\mu}$ and $\psi\star\widetilde{\mu}$ is defined as the purely discontinuous local martingale (under $P$) with jump process $\int_E U \mu(\{t\}, dx)$. If the increasing process $|\psi|\star \mu$ (or equivalently, $|\psi|\star \nu$) is locally integrable, then again, $\psi$ is integrable with respect to $\widetilde{\mu}$ and is the purely discontinuous local martingale as in the first case and we have $\psi\star \widetilde{\mu} = \psi \star \mu - \psi \star \nu$. Finally, if the process $|\psi|^2\star \nu$ is integrable, then $U$ is integrable with respect to $\widetilde{\mu}$ and $\widetilde{Z} \star \widetilde{\mu}$ is a square integrable martingale, purely discontinuous, with predictable quadratic variation $\langle \psi\star \widetilde{\mu} \rangle = |\psi|^2\star \nu$. These properties and their proofs can be found in Section II.1.d of \citet{jacod2013limit}.\\
We will assume that $W$ and $\widetilde{\mu}$ satisfy the following weak representation property with respect to $P$ and $\mathcal{F}$: Every local martingale $M$ with respect to $(\mathcal{F}, P)$ admits the following decomposition:
\begin{equation}
\label{representation}
M_t = M_0 + (\eta\cdot W)_t +(\psi \star \widetilde{\mu})_t\vcentcolon= M_0 + \int_0^t \eta_s dW_s + (\psi \star \widetilde{\mu})_t,~~~~ \forall t \ge 0,~P-a.s.
\end{equation}
where $\eta$ is a progressively measurable process and $\psi$ a predictable process such that
\begin{equation}
\int_0^T |\eta_s|^2 ds < \infty, ~~ (|\psi|^2\star \nu)_T < \infty,~P-a.s.
\end{equation}
We introduce the following spaces:
\begin{itemize}
\item $L^{\exp}$ is the space of all $\mathcal{F}_T$-measurable random variables $X$ such that:
$$ E[\exp(\gamma |X|)] < \infty,~~\forall \gamma > 0.$$
\item $\mathcal{D}_0^{\exp}$ is the space of progressively measurable processes $(X_t)_{0\le t \le T}$ with
$$ E\left[\exp\left(\gamma~ \underset{0 \le t \le T}{\esssup} |X_t|\right)\right] < \infty,~~\forall \gamma > 0.$$
\item $\mathcal{D}_1^{\exp}$ is the space of progressively measurable processes $(X_t)_{0\le t \le T}$ with
$$ E\left[\exp\left(\gamma \int_0^T |X_s| ds \right)\right] < \infty,~~\forall \gamma > 0.$$
\item $\mathcal{H}^{2, p}_\lambda$ is the space of predictable processes $\psi$ such that
$$ E\left[\left(\int_0^T |\psi|^2_{s, \lambda} ds \right)^{\frac{p}{2}}\right]^{\frac{1}{p}} < \infty,$$
where
$$|\psi|^2_{s, \lambda} \vcentcolon= \int_E |\psi_s(x)|^2 \xi_s(x)\lambda(dx).$$
\item $H^{2, p}$ is the set of all predictable processes $\eta$ such that
$$ E\left[ \left( \int_0^T |\eta_s|^2 ds\right)^{\frac{p}{2}}\right]^{\frac{1}{p}} < \infty.$$
\end{itemize}
\subsection{The optimization problem}
For every probability $Q \ll P$ on $\mathcal{F}_T$, we denote by $D = (D_t)_{0 \le t \le T}$ its Radon-Nikodym density with respect to $P$, that is,
$$D_t = E\left[\left.\frac{dQ}{dP} \right|\mathcal{F}_t\right],~~ t \ge 0.$$
$D$ is a c\`adl\`ag nonnegative $P$-martingale. Let $\tau_n \vcentcolon=\inf\{t \ge 0, D_t \le 1/n\}$ and consider the local martingale $M^n_t = \int_0^{t\wedge \tau_n} D_{s^-}^{-1} dD_s$. Thanks to the weak representation property, there exist two predictable processes $(\eta^n_s)$ and $(\psi^n_s)$, $s \le \tau_n$, such that,
$\int_0^{t \wedge \tau_n} |\eta_s^n|^2 ds < \infty$ and $(|\psi)|^2\star \nu)_{t\wedge \tau_n} < \infty$ and
$$D_{t \wedge \tau_n} = \mathcal{E}\left((\eta^n\cdot W)_. + (\psi^n\star \widetilde{\mu})_.\right)_{t \wedge \tau_n},~~ t\ge 0, ~~P-a.s.$$ 
Consistency requires that we should have $\eta_t^{n} = \eta_t^{n+1}$ $dt\otimes dP$-a.e and $ \psi^n_t(x)  = \psi^{n+1}_t(x)$ $\nu(dx) \otimes dt\otimes dP$-a.e on $\{t \le \tau_n \wedge T\}$. By the fact that $\tau_n \nearrow \infty$ $Q$-a.s., we obtain the existence of $Q$-a.s. defined predictable processes $\eta$ and $\psi$ such that,
\begin{equation}
\label{density0}
D_t = \mathcal{E}\left((\eta\cdot W)_. + (\psi\star \widetilde{\mu})_.\right)_{t},~~ t\ge 0, ~~Q-a.s.
\end{equation}
where $\int_0^T |\eta_s|^2 ds < \infty$ and $(|\psi|^2\star \nu)_T < \infty~~Q-a.s.$. Note that since for all $t \in [0, T]$, $D_t > 0~Q$-a.s., then we must have for every $t \in [0, T]$,  $\psi_t(x) > -1$ $dQ\times \nu(dx, dt)$-a.e. and we can rewrite $(D_t)$ as in the following:
\begin{equation}
\label{density}
D_t = \exp\left((\eta\cdot W)_t + (\psi\star\widetilde{\mu})_t -\frac{1}{2} \int_0^t |\eta_s|^2 ds + ((\ln(1+\psi)-\psi)\star\mu)_t\right)~~ Q-a.s.
\end{equation}
We now introduce the following time consistent penalty for a probability $Q \ll P$ on $\mathcal{F}_T$:
\begin{equation}
\label{penalty}
\gamma_t(Q) \vcentcolon= E_Q\left[\left.\int_t^T r(s, w, \eta_s, \psi_s) ds \right|\mathcal{F}_t\right],
\end{equation}
where $r:[0, T]\times \Omega \times \R^d\times L^2(E, \lambda) \to [0, +\infty]$ is a suitable measurable function that is convex and lower-semicontinuous in $(\eta, \psi)$ and such that $r(t,0,0)=0$. Note that, since $r$ is non-negative, $r$ is minimal at $\eta =0$ and $\psi=0$ and this corresponds to the probabilistic model $P$. Therefore, the reference probability has the highest plausibility.  In the following, we will consider probabilities $Q \in \mathcal{Q}^f$ where
\begin{equation}
\label{Q}
\mathcal{Q}^f=\{Q \ll P, \gamma_0(Q) < \infty\}.
\end{equation}
In order to solve the stochastic control problem with BSDEs, we need to impose some regularity and growth conditions on the penalty function. In a Brownian setting,  \citet{faidi2013robust} assumed the penalty function to be bounded from below by the relative entropy. In the same way, we will assume that there exists $\widetilde{K}_2, \widetilde{K}_1 > 0$ such that,
$$\gamma_0(Q) \ge -\widetilde{K}_2 + \widetilde{K}_1 H(Q|P).$$ 
Let $f$ be the function defined as follow:
\begin{equation}
\label{f}
f(x) = 
\left\{
\begin{aligned}
&(1+x)\log(1+x) - x,~~\text{if}~x\ge -1;\\
&\infty~~~~~~~~~~~~~~~~~~~~~~~~~~~~\text{otherwise};
\end{aligned}\right.
\end{equation}
For the latter inequality to be verified, a sufficient condition on $r$ will be the following:
\begin{enumerate}[label=($\mathcal{A}$\textsubscript{r})]
\item There exists $K_1, K_2 > 0$ such that for all $w\in \Omega$, $t\in [0, T]$, $\eta \in \R^d$ and $\psi \in L^2(\mathcal{E}, \lambda;\R)$, we have,
$$r(t,w,\eta,\psi) \ge -K_2 + K_1\left(\frac{|\eta|^2}{2} + \int_E f(\psi(x))\xi_t(w,x)\lambda(dx)\right).$$
\label{Ar}
\end{enumerate}
The following proposition shows that the entropic penalty can be retrieved with a special choice for $r$. A detailed proof is given in the Appendix \ref{Appendix1}.
\begin{Proposition}
Let $r(t, \eta, \psi) = \frac{1}{2}|\eta|^2+\int_E f(\psi(x))\xi_t(x)\lambda(dx)$ and $Q \in \mathcal{Q}_f$. Then, the penalty function corresponds to the relative entropy, that is
$$\gamma_0(Q)\underset{(\Delta)}{=}E_Q\left[\int_0^T \left(\frac{|\eta_t|^2}{2} + \int_E f(\psi_t(x))\xi_t(x)\lambda(dx)\right)dt\right] = H_0(Q|P),$$
where,
\begin{equation}
H_t(Q|P) = E_Q\left[\log\left(\left.\frac{dQ}{dP}\right)\right|\mathcal{F}_t\right].
\end{equation}
Moreover, we have for a general $r$ verifying \ref{Ar}, 
\begin{equation}
\label{entropy}
H(Q|P) \le  \frac{\gamma_0(Q)}{K_1} +\frac{TK_2}{K_1}.
\end{equation}
In particular, $H(Q|P)$ is finite for all $Q \in \mathcal{Q}_f$.
\end{Proposition}
\begin{Remark}\
Let $r^*(t, z, \widetilde{z}) = \underset{\eta \in \R^d, \psi \in L^2(\lambda)}{\sup} (z\cdot \eta + \int_E \widetilde{z}(x) \psi(x) \lambda(dx) - r(t, \eta, \psi))$ the Fenchel conjugate $r$. Assumption \ref{Ar} implies that, for $w \in \Omega$, $z \in \R^d$, $t\in [0, T]$ and $\widetilde{z} \in L^2(\mathcal{E}, \lambda;\R)$,
\begin{equation}
\label{r*}
r^*(t,w, z, \widetilde{z}) \le K_2 + \frac{|z|^2}{2K_1} + K_1\int_E f^*\left(\frac{\widetilde{z}(x)}{K_1\xi_t(x)}\right)\xi_t(x) \lambda(dx),
\end{equation}
where $f^*(x) \vcentcolon= e^x - x - 1$ is the Fenchel conjugate of the function $f$.
\end{Remark}
\noindent
Now, given a positive adapted process $\delta$, we define the discounting process:
$$ S_t^\delta \vcentcolon = \exp\left(-\int_0^t \delta_s ds\right),~~~~ 0 \le t \le T,$$
and the auxiliary quantities,
\begin{align*}
\mathcal{U}_{t,T}^\delta &\vcentcolon= \alpha \int_t^T \frac{S_s^\delta}{S_t^\delta}U_s ds + \overline{\alpha}\frac{S_T^\delta}{S_t^\delta}\overline{U}_T,~~ 0 \le t \le T,~~\alpha, \overline{\alpha} \ge 0,\\
\mathcal{R}_{t,T}^\delta(Q) &\vcentcolon=\int_t^T \frac{S_s^\delta}{S_t^\delta} r(s, \eta_s, \psi_s)ds,~~ 0 \le t \le T.
\end{align*}
Now we consider the cost functional
\begin{equation}
\label{cost}
c(w, Q) \vcentcolon=\mathcal{U}_{0,T}^\delta(w) + \beta \mathcal{R}_{0,T}(Q)(w),
\end{equation}
which consists of two terms. The first one is a discounted utility term that is the sum of a final utility $\overline{U}_T$ and a continuous utility with utility rate $(U_s)$. For instance, $(U_s)$ can be seen as the utility coming from investing/consuming and $\overline{U}_T$ as the utility coming from the terminal wealth. The second term is simply the penalty term and measure the \q{distance} between the probability $Q$ and the reference probability $P$. The parameter $\beta$ might be viewed as measuring the degree of confidence of the reference probability $P$. The higher $\beta$ is, the more confident we are in $P$, with the limiting case $\beta \uparrow \infty$ (respectively $\beta \downarrow 0$) corresponding to full degree of confidence (respectively distrust).\\
Our objective is to solve the following optimization problem:
\begin{equation}
\label{optProblem}
\text{Minimize the functional}~~Q\mapsto\Gamma(Q) \vcentcolon= E_Q[c(., Q)],
\end{equation}
over the set $\mathcal{Q}_f$.
To guarantee the well-posedeness of the problem, we will assume the following:
\begin{enumerate}[label=($\mathcal{A}$\textsubscript{u})]
\item
\begin{enumerate}[label = \roman*.]
\label{Au}
\item The discounting rate is bounded by some constant $||\delta||_\infty$;
\item The process $U$ belongs to $D_1^{\exp}$;
\item The terminal utility $\overline{U}_T$ belongs to $L^{\exp}$.
\end{enumerate}
\end{enumerate}
\begin{Remark}
\label{integrability}
Under assumption \ref{Au}, we have
\begin{equation}
\label{Lexp}
E\left[\exp\left(\lambda \int_0^T|U_s|ds+ \mu |\overline{U}_T|\right)\right] < \infty,~~\forall \lambda, \mu \in \R^+ .
\end{equation}
Indeed, using the convexity of the exponential function, we get,
\begin{align*}
E\left[\exp\left(\lambda\int_0^T|U_s|ds +\mu\overline{U}_T\right)\right] &= E\left[\exp\left(\frac{1}{2}\times 2\lambda\int_0^T|U_s|ds +\frac{1}{2}\times 2\mu\overline{U}_T\right)\right]\\
&\le E\left[\frac{1}{2}\exp\left(2\lambda \int_0^T |U_s|ds\right) +\frac{1}{2}\exp\left(2\mu|\overline{U}_T|\right)\right]\\
&= \frac{1}{2}E\left[\exp\left(2\lambda \int_0^T |U_s|ds\right)\right] + \frac{1}{2}E\left[\exp\left(2\mu|\overline{U}_T|\right)\right] < \infty,
\end{align*}
where the finiteness of the two last expectations is due to assumption \ref{Au}.
\end{Remark}
\section{Some helpful estimates and existence of optimal probability}
\label{estimates}
\subsection{Auxiliary estimates}
The main objective of this section is to prove the existence of an optimal probability $Q^*$ that minimizes the functional $\Gamma$. To achieve this, we start by proving some useful auxiliary estimates. We will adapt the steps in \citet{bordigoni2007stochastic} and the inequalities therein into our setting.
\begin{Proposition}
Under assumption \ref{Ar} and \ref{Au}, there exists a constant $C \in(0, \infty)$ which depends only on $\alpha, \overline{\alpha}, \beta, \delta, T, U, \overline{U}_T$ such that
\begin{equation}
\label{upperBoundGamma}
\Gamma(Q) \le E_Q[|c(\cdot, Q)|]\le C(1+\gamma_0(Q)),~~~\text{for all}~Q \in \mathcal{Q}_f.
\end{equation}
In particular, this shows that $\Gamma(Q)$ is well defined and finite for every $Q \in \mathcal{Q}_f$.
\end{Proposition}
\begin{proof}
The first inequality is obvious. As for the second, denoting $\mathbb{U}\vcentcolon=\alpha \int_0^T |U_t| dt + \overline{\alpha}|\overline{U}_T|$, we have for $Q \in \mathcal{Q}_f$, using the fact that, $ 0 \le S_t^\delta \le 1$, 
\begin{align}
\label{expectationOfc}
E_Q[|c(\cdot, Q)|] &\le E_Q\left[\alpha \int_0^T |U_t|dt +\overline{\alpha}|\overline{U}_T|\right] + \beta E_Q\left[\int_0^T r(t,\eta_t, \psi_t)dt\right]\\
&= E_Q[\mathbb{U}] + \beta \gamma_0(Q).
\end{align}
Fenchel inequality applied to $x \mapsto x\log(x)$, gives
\begin{equation}
\label{fenchel}
xy \le \frac{1}{\lambda}(x\log(x) + e^{\lambda y-1}),~~\forall (x, y, \lambda) \in \R_+^*\times \R_+^* \times \R^*.
\end{equation}
Therefore, using this inequality with $\lambda = 1$, we get,
\begin{align*}
E_Q[\mathbb{U}] &= E\left[\frac{dQ}{dP}\mathbb{U}\right]\\
&\le E\left[\frac{dQ}{dP}\log\left(\frac{dQ}{dP}\right)\right] + E[e^{\mathbb{U} - 1}]\\
&= H(Q|P) + E[e^{\mathbb{U} - 1}]\\
&\le \frac{\gamma_0(Q)}{K_1} +\frac{TK_2}{K_1} +e^{-1}E[e^{\mathbb{U}}],
\end{align*}
where we used \eqref{entropy} in the last inequality. Going back to \eqref{expectationOfc}, we obtain,
$$E_Q[|c(\cdot,  Q)|] \le \left(\beta + \frac{1}{K_1}\right)\gamma_0(Q) + \left(\frac{TK_2}{K_1} + e^{-1}E[e^{\mathbb{U}}]\right),$$
where the term $E[e^{\mathbb{U}}]$ is finite as pointed out in remark \ref{integrability}.
We then conclude by setting $C \vcentcolon= \max\left(\beta + \frac{1}{K_1}, \frac{TK_2}{K_1} + e^{-1}E[e^{\mathbb{U}}]\right)$.
\end{proof} 
The next result shows that $\Gamma$ is bounded from below by $\gamma_0(Q)$. This will be very useful for proving the existence of an optimal probability.
\begin{Proposition}
Assume \ref{Ar} and \ref{Au} hold. Then, there exists $C \in (0, \infty)$ depending on $\alpha, \overline{\alpha}, \beta, \delta, T, U, \overline{U}_T$ such that for all $Q \in \mathcal{Q}_f$
\begin{equation}
\label{lowerBoundGamma}
\gamma_0(Q) \le C(1 + \Gamma(Q)).
\end{equation}
In particular, we have $\underset{Q\in \mathcal{Q}_f}{\inf} \Gamma(Q)> -\infty$
\end{Proposition}
\begin{proof}
Using the same notation as in the proof of the previous proposition, we have,
$$ E_Q[\mathcal{U}_{0,T}^\delta] \ge -E_Q[\mathbb{U}].$$
For every $\lambda \in \R^*$, using the inequality \eqref{fenchel}, we get,
\begin{align*}
E_Q[\mathbb{U}] &\le \frac{1}{\lambda} H(Q|P) + \frac{1}{\lambda}E[e^{\lambda \mathbb{U}-1}]\\
&\le \frac{1}{\lambda}\left(\frac{\gamma_0(Q)}{K_1}+\frac{TK_2}{K_1}\right)+\frac{e^{-1}}{\lambda}E[e^{\lambda \mathbb{U}}],
\end{align*}
where we used \eqref{entropy} in the last inequality. On the other hand, since the discounting process is bounded from above, we have
\begin{align*}
E_Q[\mathcal{R}_{0,T}] &\underset{(\Delta)}{=} E_Q\left[\int_0^T S_t^\delta r(t, \eta_t, \psi_t) dt\right]
&\ge e^{-||\delta||_\infty T}E_Q\left[\int_0^T r(t, \eta_t, \psi_t)dt\right] = e^{-||\delta||_\infty T}\gamma_0(Q).
\end{align*}
Combining the two previous inequalities leads to the following,
$$\Gamma(Q) = E_Q[\mathcal{U}_{0,T}^\delta +\beta \mathcal{R}_{0,T}^\delta] \ge \left(\beta e^{-||\delta||_\infty T} - \frac{1}{\lambda K_1}\right)\gamma_0(Q)-\frac{TK_2}{\lambda K_1}-\frac{e^{-1}}{\lambda}E[e^{\lambda U}]$$
Choosing $\lambda$ large enough such that $\mu \vcentcolon=\beta e^{-||\delta||_\infty T} - \frac{1}{\lambda K_1} > 0$, we get the desired result by setting $C \vcentcolon= \frac{1}{\mu}\max\left(1,\frac{TK_2}{\lambda K_1}+\frac{e^{-1}}{\lambda}E[e^{\lambda U}]\right)$.
\end{proof}
The following is a direct consequence of the previous proposition and inequality \eqref{entropy}.
\begin{Corollary}
Under assumptions \ref{Ar} and \ref{Au}, there exists $K \in (0, \infty)$ such that for every $Q\in \mathcal{Q}_f$, we have the following
\begin{equation}
\label{lowerBoundGamma2}
H(Q|P) \le K(1 + \Gamma(Q)).
\end{equation}
\end{Corollary}
In the same spirit of the proof of the above proposition, we have the following estimate that is crucial in proving the existence of an optimal probability $Q^* \in \mathcal{Q}_f$.
\begin{Lemma}
\label{lemIndic}
For any $\lambda > 0$ and any measurable set $A \in \mathcal{F}_T$, we have for every $Q \in \mathcal{Q}_f$
\begin{equation}
\label{estimateIndic}
E_Q[|\mathcal{U}_{0,T}^\delta| \mathbbm{1}_A]\le \frac{\gamma_0(Q)}{\lambda K_1} + \frac{TK_2}{\lambda K_1}  +\frac{e^{-1}}{\lambda}E\left[\mathbbm{1}_A\exp(\lambda\alpha \int_0^T|U_s|ds +\lambda\overline{\alpha}|\overline{U}_T|)\right].
\end{equation}
\end{Lemma}
\begin{proof}
Using inequality \eqref{fenchel}, we have for every $\lambda > 0$ and $Q\in \mathcal{Q}_f$,
\begin{equation*}
\frac{dQ}{dP}|\mathcal{U}_{0,T}^\delta|\mathbbm{1}_A \le  \frac{1}{\lambda}\left(\frac{dQ}{dP}\log\left(\frac{dQ}{dP}\right)+e^{-1} e^{\lambda \mathbb{U}} \right)\mathbbm{1}_A.
\end{equation*}
Taking the expectation under $P$ and using \eqref{entropy}, we consequently get,
$E_Q[|\mathcal{U}_{0,T}^\delta| \mathbbm{1}_A]\le \frac{\gamma_0(Q)}{\lambda K_1}+\frac{TK_2}{\lambda K_1}+\frac{e^{-1}}{\lambda}E[\mathbbm{1}_A \exp(\lambda\mathbb{U})].$
\end{proof}
\subsection{Existence of optimal probability}
In this subsection, we prove the existence of an optimal probability $Q^*\in \mathcal{Q}_f$ using a standard Koml\'os-type argument, but before let us show two important properties of the functionals $\Gamma$ and $\gamma_0$. We will introduce the following Linderberg condition on sequences of martingales converging almost surely to $0$. This technical assumption is needed to prove the lower-semicontinuity of $\gamma_0$:
\begin{enumerate}[label=($\mathcal{A}$\textsubscript{L})]
\item Every sequence $(M_n)$ of locally square integrable martingales with the representation $dM_t^n = \eta_t^n dW_t + \int_E \psi_t^n(x) \widetilde{\mu}(dt, dx)$ converging $P$-a.s. to $0$ for each $t \in [0, T]$, verifies the following Linderberg condition:
$$ \forall \epsilon \in (0, 1], \int_0^T \int_E |\psi_t^n(x)|^2 1_{|\psi_t^n(x)|\ge \epsilon|} \nu(dt, dx) \underset{n \to \infty}{\rightarrow} 0,~~\text{in}~P-\text{Probability}$$
\label{ALind}
\end{enumerate}
\begin{Proposition}
\label{convexLsc}
Under the assumption \ref{ALind}, we have the following:
\begin{enumerate}
\item $\mathcal{Q}_f$ is a convex set and the functional $Q \in \mathcal{Q}_f \mapsto \Gamma(Q)$ is convex.
\item $\gamma_0$ is lower-semicontinuous for $L^1(P)$ convergence.
\end{enumerate}
\end{Proposition}
\begin{proof}
\begin{enumerate}
\item Let $\lambda \in (0, 1)$, $Q$, $\widetilde{Q} \in \mathcal{Q}_f$ and  $Q^\lambda \vcentcolon= \lambda Q + (1-\lambda)\widetilde{Q}$. Let $D$ and $\widetilde{D}$ denote the corresponding density processes and $(\eta, \psi)$, $(\widetilde{\eta}, \widetilde{\psi})$ the associated processes via \eqref{density0}. Consider the following processes:
\begin{align*}
\eta_t^\lambda &\vcentcolon=\frac{\lambda D_{t^-}\eta_t + (1-\lambda)\widetilde{D}_{t^-}\widetilde{\eta}_t}{D^\lambda_{t^-}}\mathbbm{1}_{D^\lambda_{t^-}> 0},\\
\psi_t^\lambda &\vcentcolon=\frac{\lambda D_{t^-}\psi_t + (1-\lambda)\widetilde{D}_{t^-}\widetilde{\psi}_t}{D^\lambda_{t^-}}\mathbbm{1}_{D^\lambda_{t^-}> 0},
\end{align*}
where $D^\lambda := \lambda D + (1-\lambda)\widetilde{D}$ is the density process of $Q^\lambda$ with respect to $P$. It is easy to see that the density $D^\lambda$ satisfies the following SDE:
$$d D_t^\lambda = D_{t^-}^\lambda \left(\eta_t^\lambda dW_t + \int_E \psi^\lambda_t(x)\widetilde{\mu}(dt, dx)\right), t\in [0, T], Q^\lambda -a.s.$$
Hence, using the convexity assumption of $r$, we have,
\begin{align*}
\gamma_0(Q^\lambda) &= E_{Q^\lambda}\left[\int_0^T r(t, \eta_t^\lambda, \psi_t^\lambda)dt\right]\\
&\le E_{Q^\lambda}\left[\int_0^T \left(\frac{\lambda D_{t^-}}{D^\lambda_{t^-}}r(t, \eta_t, \psi_t) + \frac{(1-\lambda) \widetilde{D}_{t^-}}{D^\lambda_{t^-}}r(t, \widetilde{\eta}_t, \widetilde{\psi}_t)\right)\mathbbm{1}_{D^\lambda_{t^-}> 0}dt\right]\\
&= E\left[D^\lambda_T\int_0^T \left(\frac{\lambda D_{t^-}}{D^\lambda_{t^-}}r(t, \eta_t, \psi_t) + \frac{(1-\lambda) \widetilde{D}_{t^-}}{D^\lambda_{t^-}}r(t, \widetilde{\eta}_t, \widetilde{\psi}_t)\right)\mathbbm{1}_{D^\lambda_{t^-}> 0}dt\right].
\end{align*}
Using Fubini's Theorem to interchange integral and expectation followed by conditioning on $\mathcal{F}_t$ and the martingale property of the density process $D^\lambda$, yields,
$$\gamma_0(Q^\lambda) \le E\left[\int_0^T D^\lambda_t \left(\frac{\lambda D_{t^-}}{D^\lambda_{t^-}}r(t, \eta_t, \psi_t) + \frac{(1-\lambda) \widetilde{D}_{t^-}}{D^\lambda_{t^-}}r(t, \widetilde{\eta}_t, \widetilde{\psi}_t)\right)\mathbbm{1}_{D^\lambda_{t^-}> 0}dt\right].$$
Since $D^\lambda$ is right continuous, the set $\{t \in [0, T], D^\lambda_t \neq D^\lambda_{t^-}\}$ is countable. Therefore, we have, 
\begin{align*}
\gamma_0(Q^\lambda)&= E\left[\int_0^T D^\lambda_{t^-} \left(\frac{\lambda D_{t^-}}{D^\lambda_{t^-}}r(t, \eta_t, \psi_t) + \frac{(1-\lambda) \widetilde{D}_{t^-}}{D^\lambda_{t^-}}r(t, \widetilde{\eta}_t, \widetilde{\psi}_t)\right)\mathbbm{1}_{D^\lambda_{t^-}> 0}dt\right]\\
&\le E\left[\int_0^T \lambda D_{t^-}r(t, \eta_t, \psi_t) + (1-\lambda) \widetilde{D}_{t^-}r(t, \widetilde{\eta}_t, \widetilde{\psi}_t)dt\right]\\
&= \lambda \gamma_0(Q) + (1-\lambda)\gamma_0(\widetilde{Q}) < \infty.
\end{align*}
We have showed then that $\mathcal{Q}_f$ is convex. The convexity of the functional $\Gamma$ follows readily by using the same arguments used above.
\item Let $(Q^n)$ be a sequence of probability measures that converges to $Q$ in $L^1(\Omega, P)$, i.e.,  $D_T^n \to D_T$ in $L^1(\Omega, P)$ where $D^n$ and $D$ the corresponding densities processes. Let $(\eta^n, \psi^n)$ and $(\eta, \psi)$ (resp.) be the processes given by \eqref{density0} of $D^n$ and $D$ (resp.). 
Since we know that $D^n_T$ converges to $D_T$ in $L^1(P)$, the maximal Doob's inequality
$$P(\underset{0 \le t \le T}{\sup}|D^n_t - D_t|\ge \epsilon) \le \frac{1}{\epsilon}E[|D^n_T-D_T|],~~ \forall \epsilon > 0,$$
implies that $(\underset{0 \le t \le T}{\sup}|D^n_t - D_t|)$ converges to $0$ in $P$-probability. By passing by a subsequence, we can assume that $(\underset{0 \le t \le T}{\sup}|D^n_t - D_t|)$ converges to $0$ $P-a.s.$\\
We denote $M_t^n \vcentcolon=\underset{0 \le s \le t}{\sup}|D^n_s-D_s|$ and introduce the following stopping time $\tau_n \vcentcolon=\inf\{t\in [0,T], M_t^n \ge 1\}\wedge T$. We have $M^n_{\tau_n} \le M_{\tau_n^-}^n + |D^n_{\tau_n}-D_{\tau_n}|$ and by taking expectation in the latter we get
\begin{equation}
\label{equation1}
E[M_{\tau_n}^n] \le E[M_{\tau_n^-}^n] + E[|D^n_{\tau_n}-D_{\tau_n}|].
\end{equation}
Recall that  $M_T^n \underset{n \to \infty}{\rightarrow} 0$ and since $(M_t^n)_t$ is a nondecreasing process we have $M_{\tau_n^-} \le M_T^n$ so that $M_{\tau_n^-}^n \underset{n \to \infty}{\longrightarrow} 0$. We also have by the definition of the stopping time $\tau_n$ that $M_{\tau_n^-}^n \le 1$. Hence, by the dominated convergence theorem, we obtain that 
\begin{equation}
\label{equation2}
E[M_{\tau_n^-}^n] \to 0 ~\text{as}~n \to \infty.
\end{equation}
Furthermore, 
\begin{equation}
\label{equation3}
\begin{aligned}
E[|D^n_{\tau_n}-D_{\tau_n}|] &= E[|E[D^n_{T}|\mathcal{F}_{\tau_n}]-E[D_{T}|\mathcal{F}_{\tau_n}]|]=E[|E[D^n_{T} - D_{T}|\mathcal{F}_{\tau_n}]|]\\
&\le E[E[|D^n_{T}-D_{T}||\mathcal{F}_{\tau_n}]] = E[|D^n_{T}-D_{T}|]\underset{n \to \infty}{\longrightarrow} 0.
\end{aligned}
\end{equation}
Combining \eqref{equation1}, \eqref{equation2} and \eqref{equation3}, we deduce that $M_{\tau_n}^n$ converges to $0$ in $L^1(P)$. Then, by Burkholder-Davis-Gundy's inequality, we get that $[D^n-D]_{\tau_n}^{\frac{1}{2}}$ converges to $0$ in $L^1(P)$ and a fortiori in $P$-probability. Now, as $[D^n-D]_T = [D^n-D]_{\tau_n}\mathbbm{1}_{\tau_n = T} + [D^n-D]_T \mathbbm{1}_{\tau_n < T}$, then for every $\epsilon > 0$,
\begin{align*}
P([D^n-D]_T \ge \epsilon) &\le P([D^n-D]_{\tau_n}\mathbbm{1}_{\tau_n = T} \ge \epsilon) + P([D^n-D]_T \mathbbm{1}_{\tau_n < T} \ge \epsilon)\\
&\le P([D^n-D]_{\tau_n} \ge \epsilon) + P(\tau_n < T),
\end{align*}
and 
$$P(\tau_n < T) = P(\exists t \in [0, T]~\text{s.t.}~ M^n_t \ge 1)\le P(M_T^n \ge 1)\underset{n \to \infty}{\longrightarrow} 0.$$
So, we get that $[D^n-D]_T$ converges to $0$ in $P$-probability. On the other hand, since $D_t^n - D_t \to 0$, thanks to the assumption \ref{ALind}, we get from Corollary 1 in \citet{shiryayev1981martingales} that $\langle D^n-D\rangle_T$ converges to $0$ in $P$-probability and by passing to a subsequence while keeping the same notation, we may say that $\langle D^n-D\rangle_T$ converges to $0$ $P-a.s.$. But, we know that,
$$\langle D^n-D \rangle_T = \int_0^T |D^n_{t^-}\eta_t^n - D_{t^-}\eta_t|^2 dt + \int_0^T \int_E |D^n_{t^-}\psi_t^n(x)-D_{t^-}\psi_t(x)|^2 \nu(dt, dx).$$
Therefore, we immediately obtain that $D^n_{t^-}\eta_t^n \to D_{t^-}\eta_t$ $dP\times dt-a.e.$ and $dP\times dt-a.e.$, $D^n_{t^-}\psi_t^n(x)\to D_{t^-}\psi_t(x)$ in $L^2(E, \lambda)$. Next, we will show that $\gamma_0(Q) \le \underset{n \to \infty}{\liminf}~\gamma_0(Q^n)$. Assume by way of contradiction that $\gamma_0(Q) > l\vcentcolon=\underset{n \to \infty}{\liminf}~\gamma_0(Q^n)$. By passing to a subsequence, we may assume that $\gamma_0(Q^n) \to l$. Let $\zeta \vcentcolon= \inf\{t\in[0,T], D_t = 0\}$ and $\zeta_n \vcentcolon= \inf\{t\in[0,T], D^n_t = 0\}$. Since $D^n_t = 0$ on $\{t > \zeta_n\}$, we must have $\zeta \le \underset{n \to \infty}{\liminf}~\zeta_n$. Hence, for $\epsilon \vcentcolon = \frac{\gamma_0(Q)-l}{2}$, there is $k \in \N$ such that for $T_k \vcentcolon= \zeta \wedge\{\zeta_k, \zeta_{k+1},...\}$, we have
\begin{align*}
\gamma_0(Q) &= E_{Q}\left[\int_0^T r(t, \eta_t, \psi_t)dt\right] = E\left[\int_0^\zeta D_{t^-}r(t, \eta_t, \psi_t) dt\right]\\
&\le E\left[\int_0^{T_k} D_{t^-}r(t, \eta_t, \psi_t)dt\right] + \epsilon\\
&= E\left[\int_0^{T_k} g(t, D_{t^-}, D_{t^-}\eta_t, D_{t^-}\psi_t)dt\right] + \epsilon,
\end{align*}
where $g(t,x,y,z) \vcentcolon = xr(t, \frac{y}{x}, \frac{z}{x})$. Clearly, since $r$ is lower-semicontinuous in $(\eta, \psi) $, we get that also $g$ is also lower-semicontinuous. Hence, by Fatou's lemma, we obtain
\begin{align*}
E\left[\int_0^{T_k} g(t, D_{t^-}, D_{t^-}\eta_t, D_{t^-}\psi_t)dt\right] &\le  \underset{n \ge k}{\liminf}~E\left[\int_0^{T_k} g(t, D^n_{t^-}, D^n_{t^-}\eta_t, D^n_{t^-}\psi_t)dt\right]\\
&\le \underset{n \ge k}{\liminf}~E\left[\int_0^{\zeta_k} g(t,D^n_{t^-},  D^n_{t^-}\eta_t^n, D^n_{t^-}\psi_t^n)dt\right]\\
&=\underset{n \ge k}{\liminf}~\gamma_0(Q^n)=l,
\end{align*}
so that we have $\gamma_0(Q) \le l + \epsilon < \gamma_0(Q)$ which is a contradiction.
\end{enumerate}
\end{proof}
In the next theorem, we show the existence of an optimal probability $Q^* \in \mathcal{Q}_f$.
\begin{Theorem}
\label{existence}
Assume \ref{Ar}, \ref{Au} and \ref{ALind} hold. Then there exists a probability measure $Q^*$ minimizing $Q \mapsto \Gamma(Q)$ over $\mathcal{Q}_f$.
\end{Theorem}
\begin{proof}
Let $Q^n$ a minimizing sequence in $\mathcal{Q}_f$ such that
$$\Gamma(Q^n)\underset{n \to \infty}{\searrow\searrow} \underset{Q \in \mathcal{Q}_f}{\inf}\Gamma(Q),$$
and we denote by $D^n$ the corresponding density process. Since we have $D_T^n \ge 0$, it follows from Koml\'os' lemma that there exists a sequence denoted $\overline{D}^n_T$ such that $\overline{D}^n_T \in conv(D^n_T,D_T^{n+1},...)$ for each $n \in \N$ and $(\overline{D}^n_T)$ converges $P-a.s.$ to a random variable $\overline{D}^\infty_T$. Now, we will show that $\overline{D}^\infty_T$ is associated with a probability measure $\overline{Q}^\infty$. First, we have $\overline{D}^\infty_T$ is nonnegative as the $P-a.s.$ limit of the nonnegative sequence $(\overline{D}^n_T)_n$. Second, since $\mathcal{Q}_f$ is convex, each $\overline{D}^n_T$ is associated with a probability measure $\overline{Q}^n \in \mathcal{Q}_f$. Now, thanks to the convexity of $\Gamma$ and the fact that $(\Gamma(Q^n))_n$ is decreasing, we have the following,
\begin{equation}
\label{equation4}
\Gamma(\overline{Q}^n)\le \underset{m\ge n}{\sup} \Gamma(Q^n) = \Gamma(Q^n) \le \Gamma(Q^1).
\end{equation}
Consequently, using \eqref{lowerBoundGamma2}, we get,
\begin{equation*}
\underset{n \in \N}{\sup}~E[\overline{D}^n_T \log(\overline{D}^n_T)] = \sup H(\overline{Q}^n|P)\le K(1 + \underset{n \in \N}{\sup}~\Gamma(\overline{Q}^n))\le K(1 +\Gamma(Q^1)) < \infty.
\end{equation*}
By Vall\'ee-Poussin's criterion, the sequence $(\overline{D}_T^n)$ is $P$- uniformly integrable and therefore converges to $\overline{D}^\infty_T$ in $L^1(P)$. Hence, we have, $E[\overline{D}^\infty_T] = \underset{n \to \infty}{\lim} E[\overline{D}^n_T] = 1$ since $E[\overline{D}^n_T] = 1$ for all $n\in \N$. This shows that $\overline{D}^\infty_T$ can be associated with a probability $\overline{Q}^\infty$ on $\mathcal{F}_T$ such that $d\overline{Q}^\infty = \overline{D}^\infty_T dP$. Our next step is to prove that this probability $\overline{Q}^\infty$ belongs to $\mathcal{Q}_f$. By Proposition \ref{convexLsc}, we know that $\gamma_0$ is lower-semicontinuous with respect to $L^1(\Omega, P)$ convergence. Therefore, we get since $\overline{D}_T^n \overset{L^1}{ \to } \overline{D}_T^\infty$,
$$\gamma_0(\overline{Q}^\infty) \le \underset{n \to \infty}{\liminf}~\gamma_0(\overline{Q}^n).$$
But, thanks to \eqref{lowerBoundGamma}, we know that $\gamma_0(Q) \le C(1 + \Gamma(Q))$. Consequently, we obtain that,
$$\underset{n \to \infty}{\liminf}~\gamma_0(\overline{Q}^n) \le C(1 + \underset{n \in \N}{\sup}~\Gamma(\overline{Q}^n)).$$
The RHS of the last inequality is finite thanks to \eqref{equation4}. We then conclude that $\gamma_0(\overline{Q}^\infty) < \infty$, i.e., $\overline{Q}^\infty \in \mathcal{Q}_f$. It remains to show that $\overline{Q}^\infty$ is optimal. Note that using the same arguments in the proof of Proposition \ref{convexLsc}, the function $Q \mapsto E_Q[\mathcal{R}^\delta_{0,T}(Q)]=E_Q[\int_0^T S^\delta_t r(t, \eta_t, \psi_t)dt]$ is lower-semicontinuous for $L^1(\Omega, P)$ convergence and therefore we get immediately that,
$$E_{\overline{Q}^\infty}\left[\mathcal{R}_{0,T}^\delta(\overline{Q}^\infty)\right] \le \underset{n \to \infty}{\liminf}~E_{\overline{Q}^n}\left[\mathcal{R}_{0,T}^\delta(\overline{Q}^n)\right].$$
We denote $\overline{Y}^n\vcentcolon=\overline{D}^n_T \mathcal{U}_{0,T}^\delta$ and $\overline{Y}^\infty\vcentcolon=\overline{D}^\infty_T \mathcal{U}_{0,T}^\delta$. If we prove that we also have 
\begin{equation}
\label{equation5}
E[\overline{Y}^\infty] \le \underset{n \to \infty}{\liminf}~E[\overline{Y}^n],
\end{equation}
then we would have
\begin{align*}
\Gamma(\overline{Q}^\infty)&= E[\overline{Y}^\infty] + E_{\overline{Q}^\infty}[\mathcal{R}_{0, T}^\delta(\overline{Q}^\infty)]\\
&\le \underset{n \to \infty}{\liminf}~E[\overline{Y}^n] + \underset{n \to \infty}{\liminf}~E_{\overline{Q}^n}\left[\mathcal{R}_{0,T}^\delta(\overline{Q}^n)\right]\\
&\le \underset{n \to \infty}{\liminf}~E[\overline{Y}^n]+E_{\overline{Q}^n}\left[\mathcal{R}_{0,T}^\delta(\overline{Q}^n)\right]\\
&= \underset{n \to \infty}{\liminf}~\Gamma(\overline{Q}^n) = \underset{Q \in \mathcal{Q}_f}{\inf}~\Gamma(Q),
\end{align*}
which proves that indeed $\overline{Q}^\infty$ is optimal. Although $\overline{Y}^n$ is linear in $\overline{D}^n_T$, we cannot use Fatou's lemma since ther term $\mathcal{U}_{0,T}^\delta$ has no lower bound. To remediate this, we introduce the following:
$$\widetilde{R}_m \vcentcolon=\mathcal{U}_{0,T}^\delta \mathbbm{1}_{\mathcal{U}_{0,T}^\delta \ge -m} \ge -m,~~ m\in \N.$$
Hence, we have for $n \in \N \cup \{\infty\}$,
$$\overline{Y}^n=\overline{D}^n_T\mathcal{U}_{0,T}^\delta= \overline{D}^n_T \widetilde{R}_m + \overline{D}^n_T\mathcal{U}_{0,T}^\delta \mathbbm{1}_{\mathcal{U}_{0,T}^\delta < -m}.$$
Because now $\widetilde{R}_m$ is bounded below by $-m$, we can use Fatou's lemma to get,
$$E[\overline{D}^\infty \widetilde{R}_m] \le \underset{n \to \infty}{\liminf}~E[\overline{D}^n \widetilde{R}_m].$$
Consequently, by adding and subtracting $E[\overline{D}^n_T\mathcal{U}_{0,T}^\delta \mathbbm{1}_{\mathcal{U}_{0,T}^\delta < -m}]$, we obtain,
\begin{align*}
E[\overline{Y}^\infty] &\le \underset{n \to \infty}{\liminf}~ E[\overline{D}^n \widetilde{R}_m] + E[\overline{D}^n_T\mathcal{U}_{0,T}^\delta \mathbbm{1}_{\mathcal{U}_{0,T}^\delta < -m}]\\
&\le \underset{n \to \infty}{\liminf}~E[\overline{Y}^n] + 2 \underset{n \in N \cup \{\infty\}}{\sup}~E[\mathcal{D}^n_T|\mathcal{U}_{0,T}^\delta|\mathbbm{1}_{\mathcal{U}_{0,T}^\delta < -m}].
\end{align*}
The desired inequality \eqref{equation5} will follow once we prove that
$$\underset{m \to \infty}{\lim}~\underset{n \in N \cup \{\infty\}}{\sup}E[\mathcal{D}^n_T|\mathcal{U}_{0,T}^\delta|\mathbbm{1}_{\mathcal{U}_{0,T}^\delta < -m}] = 0,$$
and this is where we use Lemma \ref{lemIndic}. Indeed, thanks to this lemma, we have,
\begin{align*}
E[\mathcal{D}^n_T|\mathcal{U}_{0,T}^\delta|\mathbbm{1}_{\mathcal{U}_{0,T}^\delta < -m}] = E_{\overline{Q}^n}[|\mathcal{U}_{0,T}^\delta|\mathbbm{1}_{\mathcal{U}_{0,T}^\delta < -m}]&\le \frac{\gamma_0(\overline{Q}^n)}{\lambda K_1} + \frac{TK_2}{\lambda K_1} + \frac{e^{-1}}{\lambda}E[\exp(\lambda \mathbb{U)}\mathbbm{1}_{\mathcal{U}_{0,T}^\delta < -m}]\\
&\le \frac{C(1+\Gamma(\overline{Q}^n))}{\lambda K_1}+\frac{TK_2}{\lambda K_1} + \frac{e^{-1}}{\lambda}E[\exp(\lambda \mathbb{U)}\mathbbm{1}_{\mathcal{U}_{0,T}^\delta < -m}].
\end{align*}
Using \eqref{equation4}, we deduce that
$$
\underset{n \in N \cup \{\infty\}}{\sup}~E[\mathcal{D}^n_T|\mathcal{U}_{0,T}^\delta|\mathbbm{1}_{\mathcal{U}_{0,T}^\delta < -m}] \le \frac{C(1 +\max(\Gamma(Q^1), \Gamma(\overline{Q}^\infty)))}{\lambda K_1}+\frac{TK_2}{\lambda K_1} + \frac{e^{-1}}{\lambda}E[\exp(\lambda \mathbb{U)}\mathbbm{1}_{\mathcal{U}_{0,T}^\delta < -m}].$$
By the dominated convergence theorem, the third term in the RHS of the previous inequality goes to $0$ as $m \to \infty$. Hence, we for all $\lambda > 0$, we have
$$\underset{m \to \infty}{\lim}~\underset{n \in N \cup \{\infty\}}{\sup}~E[\mathcal{D}^n_T|\mathcal{U}_{0,T}^\delta|\mathbbm{1}_{\mathcal{U}_{0,T}^\delta < -m}] \le \frac{C(1 +\max(\Gamma(Q^1), \Gamma(\overline{Q}^\infty)))}{\lambda K_1}+\frac{TK_2}{\lambda K_1}.$$
Sending $\lambda$ to $\infty$, we finally obtain the desired result.
\end{proof}
\section{Related BSDE with jumps}
\label{BSDE}
This section is devoted to the study of the dynamic value process $V$ associated to the optimization problem \eqref{optProblem} using stochastic control techniques. More precisely, we prove that the dynamic process is the unique solution of a certain QEBSDEJ. This extends the previous work by \citet{schroder1999optimal}, \citet{skiadas2003robust} and \citet{lazrak2003generalized}.\\
We first introduce some notations that we will use below. Let $\mathcal{S}$ denote the set of all stopping time $\tau$ with values in $[0, T]$ and $\mathcal{D}$ the space of all density processes $D^Q$ with $Q\in \mathcal{Q}_f$. We also define,
\begin{align*}
\mathcal{D}(Q, \tau) &\vcentcolon = \{ Q^\prime \in \mathcal{Q}_f,~ Q^\prime = Q~\text{on}~\mathcal{F}_\tau\},\\
\Gamma(Q, \tau) &\vcentcolon = E_Q[c(\cdot, Q)|\mathcal{F}_\tau].
\end{align*}
As in \citet{karoui1981aspects}, we define the minimal conditional cost at time $\tau$ by
$$J(Q, \tau) \vcentcolon= Q-\underset{Q^{\prime} \in \mathcal{D}(Q, \tau)}{\essinf}~\Gamma(Q^{\prime}, \tau).$$
For $Q \in \mathcal{Q}_f$ and $\tau \in \mathcal{S}$, we now define the value of the control problem starting at time $\tau$ instead of $0$ and assuming one has used the model $Q$ up to time $\tau$,
\begin{align*}
\widetilde{V}(Q^\prime, \tau) &\vcentcolon=E_{Q^\prime}[\mathcal{U}_{\tau, T}^\delta |\mathcal{F}_\tau] + \beta E_{Q^\prime}[\mathcal{R}_{\tau, T}^\delta(Q^\prime) |\mathcal{F}_\tau],\\
V(Q, \tau) &\vcentcolon= Q-\underset{Q^\prime \in \mathcal{D}(Q, \tau)}{\essinf} \widetilde{V}(Q^\prime, \tau).
\end{align*}
The following martingale optimality principle is a consequence of Theorems 1.15, 1.17 and 1.21 in \citet{karoui1981aspects}. It is the analogue of Proposition 3.4 in \citet{faidi2013robust} in a Brownian setting but the proofs also hold in our setting with obvious modifications.
\begin{Proposition}
\label{bellman}
Under \ref{Au} and \ref{Ar}, we have:
\begin{itemize}
\item The family $\{J(Q, \tau)|\tau \in \mathcal{S}, Q \in \mathcal{Q}_f\}$ is a submartingale system, that is for any $Q \in \mathcal{Q}_f$ and stopping times $\sigma \le \tau$, we have,
\begin{equation}
\label{submart}
E_Q[J(Q,\tau)|\mathcal{F}_\sigma] \ge J(Q, \sigma)~~Q-a.s.
\end{equation}
\item $\widehat{Q} \in \mathcal{Q}_f$ is optimal if and only if the family $\{J(\widehat{Q}, \tau)|\tau \in \mathcal{S}\}$ is a $\widehat{Q}$-martingale system which means that for any stopping times $\sigma \le \tau$
$$E_{\widehat{Q}}[J(\widehat{Q}, \tau)|\mathcal{F}_\sigma]=J(\widehat{Q}, \sigma)~~\widehat{Q}-a.s.$$
\item For each $Q \in \mathcal{Q}_f$, there exists an adapted RCLL process $J^Q = (J^Q_t)_{t \in [0, T]}$ which is a right closed $Q$-submartingale such that for every stopping time $\tau$
$$J_\tau^Q = J(Q, \tau)~~Q-a.s.$$
\end{itemize}
\end{Proposition}
Before stating the BSDE verified by the value process $V$, we will need to define a strong order relation on the set of increasing processes defined below.
\begin{Definition}
Let $X$ and $Y$ two increasing processes. We say that $X \preceq Y$ if the process $Y -X$ is increasing.
\end{Definition}
\begin{Theorem}
Assume assumptions \ref{Ar}, \ref{Au} and \ref{ALind} hold. If the optimal probability $\overline{Q}^\infty$ in Theorem \ref{existence} is equivalent to $P$, then there exists $Z$ and $\widetilde{Z}$ such that $(V, Z, \widetilde{Z})$ is  solution in $D_0^{\exp} \times \mathcal{H}^{2, p} \times \mathcal{H}^{2, p}_\lambda$ of the following BSDE:
\begin{equation}
\label{QEBSDEJ}
\left\{
\begin{aligned}
dV_t &= \left(\delta_t V_t - \alpha U_t + \beta r^*\left(t,\frac{Z_t}{\beta},\xi_t \frac{\widetilde{Z}_t}{\beta}\right)\right)dt  - Z_tdW_t - \int_E \widetilde{Z}_t(x)\widetilde{\mu}(dx, dt),\\
V_T &= \overline{\alpha}\overline{U}_T.
\end{aligned}
\right.
\end{equation}
\end{Theorem}
\begin{proof}
We will split the proof into three steps: First, we will prove that the value process $V$ is a $P$-special martingale, that is it can be decomposed as $V = V_0 + M^V + A^V$, where $M^V$ is a local martingale that can be written as $M^V = (Z\cdot W) + (\widetilde{Z} \star \widetilde{\mu})$ and $A^V$ a predictable finite variation process. Then, we will show that $(V, Z, \widetilde{Z})$ is a solution of the BSDE. Finally, we will prove that $(V, Z, \widetilde{Z})$ is in the required spaces. \\
\textbf{Step 1}: 
First, note that since we assumed that $\overline{Q}^\infty \sim P$,  then,
$$\underset{Q \in \mathcal{Q}_f}{\inf}\Gamma(Q) = \underset{Q \in \mathcal{Q}_f^e}{\inf}\Gamma(Q),$$
where $\mathcal{Q}_f^e \vcentcolon=\{Q \in \mathcal{Q}_f, ~Q\sim P\}$ and we define $\mathcal{D}^e(Q, \tau)$ accordingly. Hence, we will restrict our attention to probabilities $Q \in \mathcal{Q}_f^e$ and all essential infinimums  can be taken with respect to $P$ in the expression of $V(Q, \tau)$ and $J(Q, \tau)$, i.e.,
\begin{align*}
J(Q, \tau) &= P -\underset{Q^{\prime} \in \mathcal{D}^e(Q, \tau)}{\essinf}~\Gamma(Q^{\prime}, \tau),\\
V(Q, \tau) &= P-\underset{Q^\prime \in \mathcal{D}^e(Q, \tau)}{\essinf}~\widetilde{V}(Q^\prime, \tau).
\end{align*}
By Bayes' formula and the definition of $\mathcal{R}_{\tau, T}(Q^\prime)$, it is easy to see that $\widetilde{V}(Q^\prime, \tau)$ depends only on the values of the density process $D^{\prime}$ of $Q^\prime$ on $[\tau, T]$ and is therefore independent of $Q$. Hence, we can denote $V(Q, \tau)$ by $V(\tau)$. From the definition of $\mathcal{R}_{t,T}^\delta(Q^\prime)$ and $\mathcal{U}_{t,T}^\delta$, we have
\begin{align*}
\mathcal{R}_{0,T}^\delta(Q^\prime)&= \int_0^\tau S_t^\delta r(t,q_t^\prime, \psi_t^\prime)dt + S_\tau^\delta \mathcal{R}_{\tau,T}^\delta(Q^\prime),\\
\mathcal{U}_{0,T}^\delta &= \alpha \int_0^\tau S_t^\delta U_t dt + \mathcal{U}_{\tau,T}^\delta.
\end{align*}
Comparing $V(\tau)$ and $J(Q, \tau)$  yields for $Q \in \mathcal{Q}_f^e$ with density process $D^Q = \mathcal{E}((\eta\cdot W) + (\psi \star \widetilde{\mu}))$,
\begin{equation}
\label{equation6}
J(Q, \tau) = S_\tau^\delta V(\tau) + \alpha \int_0^\tau S_t^\delta U_t dt + \beta\int_0^\tau S_t^\delta r(t,\eta_t, \psi_t)dt,~ P-a.s.
\end{equation}
From the martingale optimality principle in Proposition \ref{bellman}, there exists an adapted RCLL process denoted $J^Q = (J_t^Q)_{t\in [0, T]}$ such that $J_\tau^Q = J(Q, \tau)$,  $Q-a.s.$ From \eqref{equation6}, we deduce that we can choose an adapted RCLL process $(V_t)_{t\in [0, T]}$ such that $V_\tau = V(\tau),~P-a.s.$ for all $\tau \in \mathcal{S}$. We can then rewrite \eqref{equation6} for every $Q \in \mathcal{Q}_f^e$ as,
\begin{equation}
\label{equation7}
 J_t^Q = S_t^\delta V_t + \alpha\int_0^t S_s^\delta U_s ds + \beta\int_0^t S_s^\delta r(s,\eta_s, \psi_s)dt,~dt\times dP-a.e.
\end{equation}
As the probability $P \in \mathcal{Q}_f^e$ corresponds to $\eta=0$ and $\psi=0$ and $r(t,0,0) = 0$, we get in particular for $Q=P$ in \eqref{equation7} that $J^P = S^\delta V + \alpha \int_0 S_s^\delta U_s ds$. By Proposition \ref{bellman}, $J^P$ is a $P$- submartingale  and thus we deduce that $V$ is a $P$-special semimartingale, i.e. its canonical decomposition can be written as
\begin{equation}
\label{V}
V = V_0 + M^V + A^V,
\end{equation}
where $M^V$ is a local martingale and $A^V$ is a predictable finite variation process. By the weak representation assumption, the local martingale $M^V$ can be written as:
$$M^V = -(Z\cdot W) - (\widetilde{Z}\star \widetilde{\mu}).$$
\textbf{Step 2}: We now prove that $(V, Z, \widetilde{Z})$ is a solution of QEBSDEJ in \eqref{QEBSDEJ}. Plugging \eqref{V} into \eqref{equation7} yields
$$dJ_t^Q = -\delta_t S_t^\delta V_t dt + \alpha S_t^\delta U_t dt + S_t^\delta\left(\beta r(t, \eta_t, \psi_t)dt - Z_t dW_t -\int_E \widetilde{Z}_t(x)\widetilde{\mu}(dx, dt) + dA_t^V\right).$$
For each $Q\in \mathcal{Q}_f^e$, we have, $D^Q = \mathcal{E}((\eta\cdot W) + (\psi\star\widetilde{\mu})),~P-a.s.$ and by Girsanov's theorem, we have, $dW_t^Q = dW_t - \eta_t dt$ is a $Q$ Brownian motion and $\nu^Q(dx, dt) = (1+\psi_t(x))\nu(dt, dx)$ is the compensation of $\mu$ under $Q$. Rewriting the dynamic of $J^Q$, we obtain,
\begin{multline}
dJ_t^Q = -\delta_t S_t^\delta V_t dt + \alpha S_t^\delta U_t dt +S_t^\delta\left(\beta r(t, \eta_t,\psi_t)dt-Z_t\eta_tdt - \int_E\widetilde{Z}_t(x)\psi_t(x)\nu(dx, dt) + dA_t^V\right)\\ - S_t^\delta\left(Z_tdW_t^Q +\int_E\widetilde{Z}_t(x)\widetilde{\mu}^Q(dx,dt)\right).
\end{multline}
But, we know thanks to Proposition \ref{bellman}, that for every $Q\in \mathcal{Q}_f^e$, $J^Q$ is a $Q$- submartingale and $J^{\overline{Q}^\infty}$ is a $\overline{Q}^\infty$-martingale. This means that we should have, 
\begin{align*}
dA_t^V &\ge Z_t\eta_tdt + \int_E\widetilde{Z}_t(x)\psi_t(x)\nu(dx, dt) - \beta r(t, \eta_t,\psi_t)dt+ \delta_t V_t dt-\alpha S_t^\delta U_t dt,~dt\times dQ-a.e.\\
dA_t^V &= Z_t\overline{\eta}^\infty_tdt + \int_E\widetilde{Z}_t(x)\overline{\psi}^\infty_t(x)\nu(dx, dt) - \beta r(t, \overline{\eta}^\infty_t,\overline{\psi}^\infty_t)+ \delta_t V_t dt-\alpha S_t^\delta U_t dt,~dt \times d\overline{Q}^\infty-a.e.
\end{align*}
Note that the above inequality and equality are verified $dt\times dP-a.e.$ since $Q \in \mathcal{Q}_f^e$ and by the assumption that $\overline{Q}^\infty \in \mathcal{Q}_f^e$, in which case they become equivalent to,
\begin{align}
\label{A1}
dA_t^V &\ge \underset{\eta_t \in \R^d, \psi_t \in L^2(\lambda)}{\esssup}\left(Z_t\eta_tdt + \int_E\widetilde{Z}_t(x)\psi_t(x)\nu(dx, dt) - \beta r(t, \eta_t,\psi_t)dt\right)+ \delta_t V_t dt-\alpha S_t^\delta U_t dt,~dt\times dP-a.e.\\
dA_t^V &= Z_t\overline{\eta}^\infty_tdt + \int_E\widetilde{Z}_t(x)\overline{\psi}^\infty_t(x)\nu(dx, dt) - \beta r(t, \overline{\eta}^\infty_t,\overline{\psi}^\infty_t)+ \delta_t V_t dt-\alpha S_t^\delta U_t dt,~dt \times dP-a.e.
\end{align}
By denoting 
$$r^*(t, z, \widetilde{z}) = \underset{\eta \in \R^d, \psi \in L^2(\lambda)}{\sup} (z\cdot \eta + \int_E z(x) \psi(x) \lambda(dx) - r(t, \eta, \psi)),$$
the Fenchel conjugate of $r$, equation \eqref{A1} implies that $dt\times dP-a.e.$,
\begin{equation}
\label{A2}
\begin{aligned}
dA_t^V &= \underset{\eta_t \in \R^d, \psi_t \in L^2(\lambda)}{\esssup}\left(Z_t\eta_tdt + \int_E\widetilde{Z}_t(x)\psi_t(x)\nu(dx, dt) - \beta r(t, \eta_t,\psi_t)dt\right)+ \delta_t V_t dt-\alpha S_t^\delta U_t dt\\
&= \beta r^*(t, \frac{1}{\beta}Z_t, \frac{1}{\beta} \xi_t \widetilde{Z}_t)dt + \delta_t V_tdt -\alpha S_t^\delta U_t dt\\
&= Z_t\overline{\eta}^\infty_tdt + \int_E\widetilde{Z}_t(x)\overline{\psi}^\infty_t(x)\nu(dx, dt) - \beta r(t, \overline{\eta}^\infty_t,\overline{\psi}^\infty_t)+ \delta_t V_t dt-\alpha S_t^\delta U_t dt.
\end{aligned}
\end{equation}
This shows in particular that 
\begin{equation}
\label{subdiff}
\left(\frac{Z_t}{\beta}, \frac{\widetilde{Z}_t}{\beta}\xi_t\right) \in \partial r(t, \overline{\eta}^\infty_t, \overline{\psi}^\infty_t),~~dt\times dP-a.e.
\end{equation}
Going back to equation \eqref{V} and replacing the finite variation process $A^V$ by its expression in \eqref{A2}, it follows that $(V, Z, \widetilde{Z})$ is solution of the following equation,
\begin{equation*}
\left\{
\begin{aligned}
dV_t &= \left(\delta_t V_t - \alpha U_t + \beta r^*\left(t,\frac{Z_t}{\beta},\xi_t \frac{\widetilde{Z}_t}{\beta}\right)\right)dt  - Z_tdW_t - \int_E \widetilde{Z}(x)\widetilde{\mu}(dx, dt),\\
V_T &= \overline{\alpha}\overline{U}_T.
\end{aligned}
\right.
\end{equation*}
\textbf{Step 3}: In this step, we show that the $(V, Z, \widetilde{Z}) \in D_0^{\exp} \times \mathcal{H}^{2, p}_{\lambda} \times \mathcal{H}^{2, 2}$. $V \in D_0^{\exp}$ follows as in \citet{faidi2013robust}. As for $Z$ and $\widetilde{Z}$, the proof will lean on some exponential transform. We introduce the following processes defined for $t\in [0, T]$ as:
\begin{align*}
Y_t^- &= -C V_t +  C \int_0^t (\alpha |U_s| + K_2 \beta) ds + C\int_0^t \delta_s |V_s|ds,\\
Y_t^+ &= C V_t + C\int_0^t (\alpha |U_s| + K_2 \beta) ds + C\int_0^t  \delta_s |V_s|ds,\\
K_t^- &= \exp(Y_t^-),~K_t^+ = \exp(Y_t^+),
\end{align*}
where $C=\frac{1}{K_1 \beta}$. For any $p \ge 1$, we have
\begin{align*}
\underset{t\in [0, T]}{\sup}~(K_t^{\pm})^p &= \underset{t\in [0, T]}{\sup}~\exp(pY_t^{\pm}) \le \underset{t\in [0, T]}{\sup}~\exp(pC|V_t| +  p C\int_0^t (\alpha|U_s| + K_2 \beta) ds + p C\int_0^t  \delta_s |V_s|ds)\\
&\le \exp(pC\underset{t\in [0, T]}{\sup}~|V_t| + p C\alpha\int_0^T |U_s|ds + pCK_2\beta T + pC||\delta||_\infty  T \underset{t\in [0, T]}{\sup}~|V_t|).
\end{align*}
Since $V \in D_0^{\exp}$ and $U \in D_1^{\exp}$, from the above inequality we deduce that $\underset{t\in [0, T]}{\sup}~K_t^{\pm} \in L^p(\Omega)$. We turn our attention to the process $Y^-$. Using \eqref{QEBSDEJ}, the process $Y^-$ verifies:
\begin{align*}
dY_t^- &= -CdV_t + C(\alpha |U_t|+K_2\beta)dt + C\delta_t |V_t|dt\\
&=C\left( \delta_t(|V_t|-V_t) + \alpha (|U_t|+U_t) + K_2\beta - \beta r^*\left(t, \frac{Z_t}{\beta}, \xi_t\frac{\widetilde{Z}_t}{\beta}\right)\right)dt + C Z_t dW_t + \int_E C\widetilde{Z}_t(x)\widetilde{\mu}(dt, dx)\\
&=\left(CK_2\beta -C\beta r^*\left(t, \frac{Z_t}{\beta}, \xi_t\frac{\widetilde{Z}_t}{\beta}\right) + \frac{|CZ_t|^2}{2} + \int_E f^*(C\widetilde{Z}_t(x))\xi_t(x)\lambda(dx)\right)dt +C\delta_t(|V_t|-V_t)dt \\
&~~~~~+C\alpha (|U_t|+U_t)dt+ CZ_tdW_t - \frac{|CZ_t|^2}{2}dt + \int_E C\widetilde{Z}_t(x) \widetilde{\mu}(dt, dx) - \int_E f^*(C\widetilde{Z}_t(x))\xi_t(x)\lambda(dx)dt\\
&= dI_t^- + d L_t^-,
\end{align*}
where
\begin{equation*}
\left\{
\begin{aligned}
dI_t^- &= \left(C\delta_t(|V_t|-V_t) + C\alpha (|U_t|+U_t)+CK_2\beta +C\beta r^*\left(t, \frac{Z_t}{\beta}, \xi_t\frac{\widetilde{Z}_t}{\beta}\right)+ \frac{|CZ_t|^2}{2}\right)dt,\\
&~~~~~~~~~~+ \int_E f^*(C\widetilde{Z}_t(x))\xi_t(x)\lambda(dx)dt\\
dL_t^- &= CZ_tdW_t - \frac{|CZ_t|^2}{2}dt + \int_E C\widetilde{Z}_t(x) \widetilde{\mu}(dt, dx) - \int_E f^*(C\widetilde{Z}_t(x))\xi_t(x)\lambda(dx)dt.
\end{aligned}
\right.
\end{equation*}
Thanks to inequality given in \eqref{r*}, we have the following:
$$-C\beta r^*\left(t, \frac{Z_t}{\beta}, \xi_t\frac{\widetilde{Z}_t}{\beta}\right) + CK_2 \beta+ \frac{|CZ_t|^2}{2} + \int_E f^*(C\widetilde{Z}_t(x))\xi_t(x)\lambda(dx) \ge 0,~~dt\times dP-a.e.$$
It is also easy to see, by the definition of Dol\'eans-Dade's exponential, that $\exp(L_t^-) = \mathcal{E}(M^-)_t$ where, $dM_t^- = CZ_tdW_t + \int_E (e^{C\widetilde{Z}_t(x)}-1) \widetilde{\mu}(dt, dx)$. Therefore, we obtain,
$$K_t^- \underset{(\Delta)}{=} \exp(Y_t^-) = \exp(V_0)\exp(I_t^-)\exp(L_t^-)=  \exp(V_0)\exp(I_t^-)\mathcal{E}(M^-)_t.$$
Using the integration by part formula, we get, $dK_t^- = K_{t^-}^-(dM_t^- + dI_t^-)$, which implies, that the predictable quadratic variation of $K^-$ verifies, $d\langle K^-\rangle_t = (K_{t^-}^-)^2 d\langle M^-\rangle_t$ and as a consequence, $d\langle M^-\rangle_t = \frac{1}{(K_{t^-}^-)^2}d\langle K^-\rangle_t$. Hence, 
\begin{equation}
\label{equation8}
\langle M^-\rangle_{T}\le \underset{t\in [0, T]}{\sup}~\left(\frac{1}{(K_t^-)^2}\right)\times \langle K^-\rangle_{T}.
\end{equation}
Now, we need to have an estimate for $\langle K^-\rangle$ in order to get one for $\langle M^-\rangle$. It\^{o}'s formula yields,
$$d(K_t^-)^2 = 2K_{t^-}^- dK_t^- + d[K^-]_t = 2(K_{t^-}^-)^2(dM_t^- + dI_t^-)+ d[K^-]_t.$$
Taking a sequence of stopping times $(T_n)$ such that for each $n\in \N$, $(\int_0^{t\wedge T_n} 2 (K_{s^-}^-)^2dM_s^-)_t$ is a uniformly integrable martingale and integrating the above equation between a stopping time $\sigma \le T$ and $T \wedge T_n$, we get,
$$[K^-]_{T\wedge T_n} - [K^-]_\sigma = (K_{T \wedge T_n}^-)^2-(K_\sigma^-)^2 - 2 \int_\sigma^{T \wedge T_n}  (K_{t^-}^-)^2(dM_t^- + dI_t^-).$$
Since $ \int_0^{T_n}  (K_{t^-}^-)^2dI_t^- \ge 0$, by taking conditional expectations on both sides, we obtain,
$$E[\langle K^-\rangle_{T\wedge T_n} - \langle K^-\rangle_\sigma|\mathcal{F}_\sigma] = E[[K^-]_{T\wedge T_n} - [K^-]_\sigma|\mathcal{F}_\sigma] \le E[(K_{T\wedge T_n}^-)^2|\mathcal{F}_\sigma]\le E[\underset{t\in [0, T]}{\sup}~(K_t^-)^2|\mathcal{F}_\sigma].$$
Finally, passing to the limit as $n \to + \infty$ and using the Monotone Convergence theorem, we have,
$$E[\langle K^-\rangle_{T}-\langle K^-\rangle_\sigma|\mathcal{F}_\sigma] \le E[\underset{t\in [0, T]}{\sup}~(K_t^-)^2|\mathcal{F}_\sigma].$$
Now, since for every $p\ge 1$, $\underset{t\in [0, T]}{\sup}~K_t^- \in L^P(\Omega)$, it follows from Garcia and Neveu Lemma (see for example Lemma 4.3 in \citet{barrieu2013monotone} or \citet{neveu1972martingales}) that 
\begin{equation}
\label{equation9}
E[\langle K^-\rangle_{T}^p] < \infty,~~\forall p \ge 1.
\end{equation}
With the same arguments used to show that $\underset{t\in [0, T]}{\sup}~K_t^- \in L^p(\Omega)$, we have also that $\underset{t\in [0, T]}{\sup}~\frac{1}{K_t^-} \in L^p(\Omega)$ for any $p \ge 1$. From \eqref{equation8} and \eqref{equation9} together with Cauchy-Schwartz inequality, we deduce that
\begin{equation}
\label{equation10}
E[\langle M^-\rangle_{T}^p]   < \infty,~~\forall p \ge 1.
\end{equation}
As for the process $Y^+$, it verifies, $dY^+_t = dI_t^+ + dL_t^+$ where,
\begin{equation*}
\left\{
\begin{aligned}
dI_t^+ &= \left(C\delta_t(|V_t|+V_t) + C\alpha (|U_t|-U_t)+CK_2\beta +C\beta r^*\left(t, \frac{Z_t}{\beta}, \xi_t\frac{\widetilde{Z}_t}{\beta}\right)+ \frac{|CZ_t|^2}{2}\right)dt\\
&~~~~~~~~~~+ \int_E f^*(-C\widetilde{Z}_t(x))\xi_t(x)\lambda(dx)dt,\\
dL_t^+ &= -CZ_tdW_t - \frac{|CZ_t|^2}{2}dt - \int_E C\widetilde{Z}_t(x) \widetilde{\mu}(dt, dx) - \int_E f^*(-C\widetilde{Z}_t(x))\xi_t(x)\lambda(dx)dt.
\end{aligned}
\right.
\end{equation*}
As $r^*$ and $f^*$ are positive functions, the process $I_t^+$ is increasing and as previously, by easy calculations, we can see that $\exp(L_t^+)=\mathcal{E}(M^+)_t$ where $dM^+_t = -CZ_tdW_t +\int_E(e^{-C\widetilde{Z}_t(x)}-1)\widetilde{\mu}(dt, dx)$. Going through the same lines as with $Y^-$, we obtain,
\begin{equation}
\label{equation11}
E[\langle M^+\rangle_{T}^p] < \infty,~~\forall p \ge 1.
\end{equation}
But, expressing the expression of predictable quadratic variation of $M^+$ and $M^-$, we get,
\begin{align*}
E\left[ \left(\int_0^T |CZ_t|^2 dt + \int_0^T\int_E (e^{C\widetilde{Z}_t(x)} - 1)^2\nu(dt, dx)\right)^p\right] < \infty,~~p \ge 1,\\
E\left[ \left(\int_0^T |CZ_t|^2 dt + \int_0^T\int_E (e^{-C\widetilde{Z}_t(x)} - 1)^2\nu(dt, dx)\right)^p\right] < \infty,~~p \ge 1.
\end{align*}
This implies from one hand that,
$$ E\left[ \left(\int_0^T |Z_t|^2 dt\right)^p\right] < \infty,~~p \ge 1,$$
and from the other hand, using the fact that $|y|^2 \le 2(|e^y-1|^2 + |e^{-y}-1|^2),~y \in \R$, we get that,
$$ E\left[ \left(\int_0^T \int_E |\widetilde{Z}_t(x)|^2 \nu(dt,dx)\right)^p \right] < \infty,~~ p\ge 1.$$
In conclusion, we have showed that $Z \in \mathcal{H}^{2, p}$ and $\widetilde{Z} \in \mathcal{H}^{2, p}_\lambda$.
\end{proof}
In the next proposition, we establish a comparison theorem for the class of BSDEs in \eqref{QEBSDEJ}. For two random variables, we write $A \le B$ if $A \le B~P-a.s.$ and for two processes $X$ and $Y$, we write $X \le Y$ if $\forall t \in [0, T]$,  $X_t \le Y_t$ $P-a.s.$ Finally, we write $(A, X) \le (B, Y)$ if $A \le B$ and $X \le Y$.
\begin{Proposition}
\label{comparison}
Assume that for $k=1,2$, $(V^k, Z^k, \widetilde{Z}^k)$ is a solution of the BSDE \eqref{QEBSDEJ} in $D_0^{\exp}\times \mathcal{H}^{2,p}\times \mathcal{H}^{2,p}_\lambda$ associated with $(U^k, \overline{U}_T^k)$. If $(U^1, \overline{U}_T^1) \le (U^2, \overline{U}_T^2)$, then,
$$\forall t \in [0, T], ~~ V_t^1 \le V_t^2~~ P-a.s.$$
\end{Proposition}
\begin{proof}
In general, establishing comparison theorems for BSDEs is obtained through an estimate of the quantity $((V^1-V^2)^+)^2$. Here, in order to take advantage of the convexity of the finite variation part of the BSDE, we will rather estimate $V^1 - \theta V^2$ for each $\theta \in (0, 1)$. Similar idea was used in \citet{briand2008quadratic} for the continuous case.\\
Let $\theta \in (0, 1)$ and $V^\theta = V^1 -\theta V^2$. We define accordingly $Z^\theta$, $\widetilde{Z}^\theta$, $U^\theta$ and $\overline{U}_T^\theta$. From \eqref{QEBSDEJ}, the dynamics of the process $V^\theta$ discounted are given by
\begin{align*}
dS_t^\delta V^\theta_t &= S_t^\delta \left[-\alpha U_t^\theta + \beta \left(r^*\left(t, \frac{Z_t^1}{\beta}, \xi_t \frac{\widehat{Z}_t^1}{\beta}\right) -\theta r^*\left(t, \frac{Z_t^2}{\beta}, \xi_t \frac{\widehat{Z}_t^2}{\beta}\right)\right)\right]dt\\
&~~~~~~~~~~~~-S_t^\delta Z_t^\theta dW_t -\int_E S_t^\delta \widetilde{Z}_t^\theta(x) \widetilde{\mu}(dx, dt)\\
&= S_t^\delta (-\alpha U_t^\theta + \beta (r^{*,1} -\theta r^{*,2}))dt-S_t^\delta Z_t^\theta dW_t -\int_E S_t^\delta \widetilde{Z}_t^\theta(x) \widetilde{\mu}(dx, dt),
\end{align*}
where, to alleviate the notations, we have denoted, $r^{*,i} = r^*(t, \frac{Z_t^i}{\beta}, \xi_t \frac{\widetilde{Z}_t^i}{\beta})$. Now, since $r^*$ is convex, the term  $r^{*,1} -\theta r^{*,2}$ can be bounded from above. Indeed,
\begin{equation}
\label{equation17}
\begin{aligned}
r^{*,1} = r^*\left(t, \frac{Z_t^1}{\beta}, \xi_t \frac{\widehat{Z}_t^1}{\beta}\right) &\le \theta r^*\left(t, \frac{Z_t^2}{\beta}, \xi_t \frac{\widehat{Z}_t^2}{\beta}\right) + (1-\theta) r^*\left(t, \frac{Z^\theta_t}{\beta(1-\theta)}, \xi_t\frac{\widetilde{Z}^\theta_t}{\beta(1-\theta)}\right) \\
&\vcentcolon= \theta r^{*,2} + (1-\theta) r^{*,\theta}.
\end{aligned}
\end{equation}
Moreover, thanks to \eqref{r*}, we have
$$r^{*,\theta} \le K_2 + \frac{|Z_t^\theta|^2}{2K_1 \beta^2(1-\theta)^2} + K_1 \int_E f^*\left(\frac{\widetilde{Z}^\theta_t(x)}{K_1\beta(1-\theta)}\right)\xi_t(x)\lambda(dx).$$
Using this last inequality in \eqref{equation17}, we get that,
\begin{equation}
\label{equation18}
\beta(r^{*,1} - \theta r^{*,2}) \le K_2 \beta (1-\theta) + \frac{|Z_t^\theta|^2}{2K_1 \beta(1-\theta)} + K_1 \beta (1-\theta) \int_E f^*\left(\frac{\widetilde{Z}^\theta_t(x)}{K_1\beta(1-\theta)}\right)\xi_t(x)\lambda(dx).
\end{equation}
To get rid of the quadratic and exponential terms in the inequality above, we will use an exponential change of variables. More precisely, let $c$ be a negative constant (to be specified later), and set $P_t = \exp(cS_t^\delta V_t^\theta)$, $Q_t = cS_t^\delta P_{t^-} Z_t^\theta$ and $\widetilde{Q}_t =cS_t^\delta P_{t^-} \widetilde{Z}_t^\theta$ . Using It\^{o}'s formula, we deduce,
\begin{align*}
dP_t &= P_{t^-} \left[c d (S_t^\delta V_t^\theta) + \frac{c^2}{2} d \langle S^\delta V^\theta \rangle_t + \int_E  f^*(-cS_t^\delta\widetilde{Z}^\theta_t(x))\mu(dx, dt)\right]\\
&= cP_{t^-} \left[ S_t^\delta (-\alpha U_t^\theta + \beta (r^{*,1} -\theta r^{*,2}))dt-S_t^\delta Z_t^\theta dW_t -\int_E S_t^\delta \widetilde{Z}_t^\theta(x) \widetilde{\mu}(dx, dt)\right.\\
&~~~~~~ \left.+ \frac{c}{2} |S_t^\delta Z_t^\theta|^2 dt + \frac{1}{c}\int_E  f^*(-cS_t^\delta\widetilde{Z}^\theta_t(x))\mu(dx, dt)\right]\\
&= cS_t^\delta P_{t^-} \left[  -\alpha U_t^\theta +\beta (r^{*,1} -\theta r^{*,2}) + \frac{c}{2}S_t^\delta |Z_t^\theta|^2 + \frac{1}{cS_t^\delta}\int_E f^*(-cS_t^\delta\widetilde{Z}^\theta_t(x))\xi_t(x) \lambda(dx))\right]dt\\
&~~~~~~ -   Q_t dW_t - \int_E  \widetilde{Q}_t(x) \widetilde{\mu}(dx, dt) + P_{t^-}\int_E  f^*(-cS_t^\delta\widetilde{Z}^\theta_t(x))\widetilde{\mu}(dx, dt)\\
&\vcentcolon= G_t dt - Q_t dW_t -\int_E \widetilde{Q}_t(x)\widetilde{\mu}(dx, dt)+ P_{t^-}\int_E  f^*(-cS_t^\delta\widetilde{Z}^\theta_t(x))\widetilde{\mu}(dx, dt).
\end{align*}
Thanks to equation \eqref{equation18}, the $G_t$ term is bounded from above,
\begin{equation}
\label{equation19}
\begin{aligned}
G_t &\le cS_t^\delta P_{t^-}\left[ -\alpha U_t^\theta + K_2\beta(1-\theta) + \frac{|Z_t^\theta|^2}{2}\left(\frac{1}{K_1 \beta(1-\theta)}+cS_t^\delta\right)\right.\\
&~~~~~~+ \left.\int_E\left( K_1 \beta (1-\theta)f^*\left(\frac{\widetilde{Z}^\theta_t(x)}{K_1\beta(1-\theta)}\right)- \frac{1}{-cS_t^\delta} f^*(-cS_t^\delta\widetilde{Z}^\theta_t(x))\right)\xi_t(x) \lambda(dx)\right]\\
&\vcentcolon= cS_t^\delta P_{t^-}\left[-\alpha U_t^\theta + K_2\beta(1-\theta) + \frac{|Z_t^\theta|^2}{2}\left(\frac{1}{K_1 \beta(1-\theta)}+cS_t^\delta\right)\right.\\
&~~~~ \left.+ \int_E \left(h(K_1\beta(1-\theta),\widetilde{Z}^\theta_t(x)) -  h(\frac{-1}{cS_t^\delta},\widetilde{Z}^\theta_t(x))\right)\xi_t(x)\lambda(dx)\right],
\end{aligned}
\end{equation}
where $h:\R\times \R^d \to \R$ defined as $h(x,z) \vcentcolon=xf^*(z/x) = xe^{z/x}-x - z$. We need to choose $c$ such that the term next to $|Z_t^\theta|^2$ is negative, that choose $c$ such that,
$$ \frac{1}{K_1 \beta(1-\theta)} \le -c S_t^\delta.$$
Since $S_t^\delta \ge e^{||\delta||_\infty T}$, it is sufficient to set $c(\theta) \vcentcolon= -\frac{e^{-||\delta||_\infty T}}{K_1 \beta(1-\theta)}$. Computing the derivative of $h$ with respect to $x$, we get, $\partial_x h(x, z) = e^{z/x}-(z/x)e^{z/x}-1$. Studying the sign of the function $x \to e^x - xe^x - 1$ by calculating its derivative, we obtain that $e^x - xe^x - 1 \le 0$, $\forall x \in \R$. Therefore, we deduce that $\partial_x h(x, z) \le 0$, $\forall x \in \R$, that $h$ is decreasing. Hence, going back to \eqref{equation19}, we get that,
\begin{equation}
G_t \le c(\theta)S_t^\delta P_{t^-}(-\alpha U_t^\theta + K_2 \beta (1-\theta))
\le S_t^\delta P_{t^-}e^{-||\delta||_\infty T}\left(\frac{\alpha}{K_1 \beta}U_t^1 - \frac{K_2}{K_1}\right) ,
\end{equation}
where we have used, in the second inequality, the fact that,
$$U_t^\theta = U_t^1 - \theta U_t^2 = \theta (U_t^1-U_t^2) + (1-\theta)U_t^1 \le (1-\theta)U_t^1.$$
Finally, denoting $D_t = \exp\left(-e^{-||\delta||_\infty T}\int_0^t S_s^\delta \left(\frac{\alpha}{K_1 \beta}U_t^1 - \frac{K_2}{K_1}\right)ds\right)$, and introducing $P_t^D\vcentcolon = D_t P_t$, $Q^D_t \vcentcolon = Q_t D_t$ and $\widetilde{Q}^D_t \vcentcolon = \widetilde{Q}_t D_t$. Using again It\^{o}'s formula, for any stopping time $0\le t \le \tau \le T$,
$$P^D_t \ge P^D_\tau + \int_t^\tau Q^D_s dW_s + \int_t^\tau \widetilde{Q}^D_s(x) \widetilde{\mu}(dx, ds)- \int_t^\tau \int_EP_{s^-}D_s f^*(-cS_s^\delta\widetilde{Z}^\theta_s(x))\widetilde{\mu}(dx, ds).$$
Considering a localizing sequence of stopping time $\tau_n$, such that the local martingales, in the above inequality, stopped in $\tau_n$ are martingales, we obtain,
$$P_t \ge E\left[\left.P_ {\tau_n} \exp\left(-e^{-||\delta||_\infty T}\int_t^{\tau_n} S_s \left(\frac{\alpha}{K_1 \beta}U_t^1 - \frac{K_2}{K_1}\right)ds\right)\right|\mathcal{F}_t\right].$$
In view of the integrability assumptions on $U^1$ and on $V$, by the dominated convergence theorem, we can deduce that,
$$P_t \ge E\left[\left.P_ {T} \exp\left(-e^{-||\delta||_\infty T}\int_t^{T} S_s \left(\frac{\alpha}{K_1 \beta}U_t^1 - \frac{K_2}{K_1}\right)ds\right)\right|\mathcal{F}_t\right].$$
But by definition of $P$, $P_T = \exp(c(\theta)S_T^\delta V_T^\theta) = \exp(c(\theta)S_T^\delta (\overline{U}_T^1-\theta\overline{U}_T^2))$, and because $\overline{U}_T^1 \le \overline{U}_T^2$ and $c(\theta)$ is negative, we get, 
$$c(\theta)S_T^\delta V_T^\theta \ge -\frac{e^{-||\delta||_\infty T}}{K_1\beta}S_T^\delta\overline{U}_T^1.$$
Therefore, we have,
$$P_t \ge E\left[\left.\exp\left(-e^{-||\delta||_\infty T} \left(\frac{S_T^\delta}{K_1\beta}\overline{U}_T^1 + \int_t^{T} S_s \left(\frac{\alpha}{K_1 \beta}U_t^1 - \frac{K_2}{K_1}\right)ds\right)\right)\right|\mathcal{F}_t \right],$$
which implies that,
$$V_t^\theta \le -\frac{K_1\beta(1-\theta)e^{||\delta||_\infty T}}{S_t^\theta}\ln E\left[\left.\exp\left(-e^{-||\delta||_\infty T} \left(\frac{S_T^\delta}{K_1\beta}\overline{U}_T^1 + \int_t^{T} S_s \left(\frac{\alpha}{K_1 \beta}U_t^1 - \frac{K_2}{K_1}\right)ds\right)\right)\right|\mathcal{F}_t \right].$$
Taking the limit when $\theta \nearrow 1$, we finally get,
$$V_t^1 \le V_t^2.$$
\end{proof}
The following corollary is a direct consequence of the comparison result above.
\begin{Corollary}
Under assumptions \ref{Ar} and \ref{Au}, the BSDE \eqref{QEBSDEJ} has a unique solution $(V, Z, \widetilde{Z})\ D_0^{\exp}\times \mathcal{H}^{2,p} \times \mathcal{H}^{2,p}_\lambda$.
\end{Corollary}
\section{Appendix}
\label{Appendix1}
\begin{Lemma}
\label{lem1}
Let $r(t, \eta, \psi) = \frac{1}{2}|\eta|^2+\int_E f(\psi(x))\xi_t(x)\lambda(dx)$ and $Q \in \mathcal{Q}_f$. Then, the following processes,
$$ M_t = \int_0^t \eta_s dW_s^Q,~~ M_t^\prime = \int_0^t\int_E \log(1+\psi_s(x))\widetilde{\mu}^Q(ds, dx),$$
are $Q$-martingales.
\end{Lemma}
\begin{proof}
Since $Q \in \mathcal{Q}_f$, we have,
\begin{equation}
\label{equation12}
E_Q\left[\int_0^T r(t, \eta_t, \psi_t) dt\right] = E_Q\left[\int_0^T \left(\frac{1}{2}|\eta_t|^2+\int_E f(\psi_t(x))\xi_t(x)\lambda(dx)\right)dt\right] < \infty.
\end{equation}
In particular, $E_Q\left[\int_0^T \frac{1}{2}|\eta_t|^2 dt \right] < \infty$, which implies that $M$ is $Q$- martingale. Now, we prove that $M^\prime$ is also a $Q$-martingale. First, note that,
$$f(x) = (1+x)\log(1+x) - x \ge \frac{1}{6}(1+x)\log^2(1+x) \ge 0,~\text{for}~ -1 \le x \le e^2-1.$$
Hence, as the RHS of \eqref{equation12} is finite, we get that
\begin{equation}
\label{equation13}
\log(1+\psi_s(x))\mathbbm{1}_{\psi_s(x) \le e^2-1} \in L^2(dQ\times \nu^Q(ds, dx)).
\end{equation}
Moreover, for $x > e^2-1$, we have,
$$(1+x)\log(1+x) \le 2((1+x)\log(1+x)-x).$$
Again, as the RHS of \eqref{equation12} is finite, we get that,
\begin{equation}
\label{equation14}
\log(1 +\psi_s(x))\mathbbm{1}_{\psi_s(x) > e^2-1} \in L^1(dQ\times \nu^Q(ds, dx)).
\end{equation}
From \eqref{equation13} and \eqref{equation14}, and using Theorem 1.8(i) in \citet{jacod2013limit}, we obtain that $M^\prime$ is $Q$-martingale.
\end{proof}
\begin{Proposition}
\label{specialCase}
Let $r(t, \eta, \psi) = \frac{1}{2}|\eta|^2+\int_E f(\psi(x))\xi_t(x)\lambda(dx)$ and $Q \in \mathcal{Q}_f$. Then, we have,
\begin{equation}
\label{equation15}
H(Q|P) = E_Q\left[\int_0^T r(t, \eta_t, \psi_t) dt\right] = E_Q\left[\int_0^T \left(\frac{1}{2}|\eta_t|^2+\int_E f(\psi_t(x))\xi_t(x)\lambda(dx)\right)dt\right].
\end{equation}
\end{Proposition}
\begin{proof}
Let $Q \in \mathcal{Q}_f$ with corresponding $(\eta, \psi)$. We introduce the following sequence of processes $(\psi_m)_{m\in \N^*}$ defined as:
$$\psi_{m,s}(x) = \psi_s(x) \mathbbm{1}_{\psi_s(x) \le m} \mathbbm{1}_{|x| \ge 1/m}.$$ It is clear that $\psi_m \in L^2(dQ\times \nu^Q(ds, dx))$. Developing the logarithm of Radon-Nikodym derivative of $Q$ w.r.t $P$ gives $Q$-a.s.:
\begin{equation}
\label{equation16}
\begin{aligned}
\log\left(\frac{dQ}{dP}\right) &= \underset{m\to \infty}{\lim} \log\left\{\mathcal{E}\left(\int_0^T \eta_s dW_s + \int_0^T \int_E \psi_{m,s}(x)\widetilde{\mu}(ds, dx)\right)\right\}\\
&= \underset{m\to \infty}{\lim} \left\{\int_0^T \eta_s dW_s -\frac{1}{2}\int_0^T |\eta_s|^2 ds + \int_0^T\int_E \psi_{m,s}(x)\widetilde{\mu}(ds, dx)\right.\\
&~~~~\left. + \int_0^T \int_E(\log(1 + \psi_{m,s}(x)) - \psi_{m,s}(x))\mu(ds, dx)\right\}\\
&=\underset{m\to \infty}{\lim} \left\{ \int_0^T \eta_s dW_s^Q + \frac{1}{2}\int_0^T |\eta_s|^2 ds + \int_0^T \int_E \psi_{m,s}(x) \widetilde{\mu}^Q(ds, dx)\right.\\
&~~~~\left. + \int_0^T \int_E (\log(1+\psi_{m,s}(x))-\psi_{m,s}(x))\widetilde{\mu}^Q(ds, dx)\right.\\
&~~~~+\left.\int_0^T \int_E\left[\psi_{m,s}(x)\psi_s(x) + (1 + \psi_s(x))\left(\log(1+\psi_{m,s}(x))-\psi_{m,s}(x)\right)\nu(ds, dx)\right]\right\}\\
&= \underset{m\to \infty}{\lim} \left\{\int_0^T \eta_s dW_s^Q + \frac{1}{2}\int_0^T |\eta_s|^2 ds + \int_0^T \int_E \log(1+\psi_{m,s}(x))\widetilde{\mu}^Q(ds, dx)\right.\\
&~~~~\left. +\int_0^T \int_E \left[(1 + \psi_s(x))\log(1+\psi_{m,s}(x))-\psi_{m,s}(x)\right]\nu(ds, dx)\right\},
\end{aligned}
\end{equation}
where we used from the second to the third inequality that, by the definition of the process $\psi_m$, $(1+\psi_s)(\log(1+\psi_{m,s})-\psi_{m,s}) \in L^1(\nu(ds, dx))$ and $\psi_{m}\psi \in L^1(\nu(ds, dx))$. In particular, $\int_0^t \int_E (\log(1 + \psi_{m,s}(x))- \psi_{m,s}(x) \widetilde{\mu}^Q(ds, dx)$ is a well defined. Lemma \ref{lem1} above insures that the following processes:
$$ M_t \vcentcolon = \int_0^t \eta_s dW_s^Q, ~~M_{t}^\prime \vcentcolon= \int_0^t \int_E \log(1 + \psi_s(x)) \widetilde{\mu}^Q(ds, dx), ~~~M_{m, t}^\prime \vcentcolon= \int_0^t \int_E \log(1 + \psi_{m,s}(x)) \widetilde{\mu}^Q(ds, dx),$$
are $Q$-martingales. Moreover, we also have that $M_{m,T}^\prime$ converges to $M^\prime_T$ in $L^1(Q)$. Indeed, decomposing $M^\prime$ and $M^\prime_m$ as in Lemma \ref{lem1} in the following way,
$$M^\prime_{m,t} = \widehat{M}_{m,t} + \widetilde{M}_{m, t},$$ 
where,
$$ \widehat{M}_{m,t} \vcentcolon= \int_0^t \int_E\log(1 +\psi_{s,m}(x))\mathbbm{1}_{\psi_{s,m}(x) \le e^2-1},~~\widetilde{M}_{m,t} \vcentcolon=\int_0^t \int_E \log(1 +\psi_{s,m}(x))\mathbbm{1}_{\psi_{s,m}(x) > e^2-1}.$$
We have the positive (resp. negative) part of $\log(1 +\psi_{m,s}(x))\mathbbm{1}_{\psi_{s,m}(x) \le e^2-1}$ increases (resp. decreases) to the positive (resp. negative) part of $\log(1 +\psi_s(x))\mathbbm{1}_{\psi_{s}(x) \le e^2-1}$ as $m$ goes to infinity and by the monotone convergence we deduce that $\widehat{M}_{m, T}$ convergence to $\hat{M}_T$ in $L^1(Q)$. Using the same argument, we have also $\widetilde{M}_{m, T}$ converges to $\widetilde{M}_T$ in $L^1(Q)$ as m goes to infinity. Hence, we obtain that $M^\prime_{m,T}$ converges to $M^\prime_T$ as $m \to \infty$. By passing to a subsequence, we may assume that $M^\prime_{m,T}$ converges to $M^\prime_T$ $Q$-a.s. Finally, by Monotone convergence theorem, the last term in \eqref{equation16} also converges to $\int_0^T \int_E f(\psi_s(x)) \nu(ds, dx)$. Consequently, \eqref{equation16} becomes,
\begin{align*}
\log\left(\frac{dQ}{dP}\right) &=  \int_0^T \eta_s dW_s^Q + \frac{1}{2}\int_0^T |\eta_s|^2 ds + \int_0^T \int_E \log(1+\psi_{s}(x))\widetilde{\mu}^Q(ds, dx)\\
&~~~~ +\int_0^T \int_E f(\psi_s(x))\nu(ds, dx).
\end{align*}
Taking the expectation under $Q$ and using the fact that $M$ and $M^\prime$ are martingales yields \eqref{equation15}.
\end{proof}

\bibliographystyle{apa}
\bibliography{biblio}

 \end{document}